\documentclass[article,a4]{elsarticle}
\usepackage{latexsym,amssymb,amsmath,amsthm,amsfonts,graphicx,float}
\usepackage{stix}
\usepackage{tikz}
\usepackage{matlab-prettifier}
\allowdisplaybreaks
\usepackage{hyperref}
\usepackage{array}
\usepackage{colortbl}
\usepackage{hhline}
\usepackage{soul}
\usepackage{color}
\usepackage[english]{babel}	
\usepackage{enumitem}
\usepackage{listings}
\tikzstyle{vertex}=[circle, draw, inner sep=0pt, minimum size=4pt]
\newcommand{\vertex}{\node[vertex]}
\newcommand{\meet}{\wedge}
\newcommand{\join}{\vee}

\newcommand{\tr}{\mathrm{tr}}

\usepackage{tikz-cd}
\begin{document}
\parskip 10pt
\parindent 0cm

\title{Pseudo triangular norms on bounded trellises}

\author{Abdelkrim Mehenni}
\ead{amehenni1@usthb.dz}
\cortext[co1]{Corresponding author}
\address{LA3C Laboratory, Mathematics Faculty, Department of Algebra and Number Theory,
University of Science and Technology Houari Boumediene, Bp 32, Bab Ezzouar, Algiers,
Algeria}
\author{Lemnaouar Zedam}
\ead{lemnaouar.zedam@univ-msila.dz}
\address{Laboratory of Pure and Applied Mathematics, Department of Mathematics, University of M'sila, Algeria}

\newtheorem{theorem}{Theorem}[section]
\newtheorem{lemma}{Lemma}[section]
\newtheorem{corollary}{Corollary}[section]
\newtheorem{proposition}{Proposition}[section]
\newtheorem{problem}{Problem}[section]
\theoremstyle{definition}
\newtheorem{definition}{Definition}[section]
\newtheorem{example}{Example}[section]
\newtheorem{remark}{Remark}[section]
\newtheorem{notation}{Notation}[section]

\begin{abstract}
In this paper, we introduce the notion of  pseudo-t-norm  on bounded trellises (also known as weakly associative lattices) as an extension  of meet and join operations (resp. t-norm)  on  bounded trellises, and provide some basic examples. 
  We provide a first generic construction method that allows extending a pseudo-t-norm on bounded trellises.  
Also, we introduce the notion of T-distributivity for any pseudo-t-norm $T$ on  bounded trellises. 
Moreover,  We determine the relationship between  pseudo-t-norms and isomorphisms on  bounded trellises.

\end{abstract}

\maketitle
\section{Introduction}
Binary operations play an important role in many of the technological tasks
scientists are faced with nowadays. They are specifically important in many problems
related to the fusion of information. More generally, binary operations are
widely used in pure mathematics (e.g., group theory, monoids theory) (see, e.g., \cite{dummit2004abstract,kolman2003discrete,lidl2012applied}). Binary operations have become essential tools in the unit interval and lattices and its applications, several notions and properties (see, e.g., \cite{roman2008lattices,yettou1}).

Triangular norms (t-norms) (as specific binary operations) were introduced by Karl Menger \cite{fried1973some} with the goal of constructing metric spaces using probabilistic
distributions (and therefore values in the interval $[0, 1]$), instead
of using real numbers, to describe the distance between two
elements. Besides, the original proposal is very weak even
including triangular conorms (t-conorms). However, only with the work of Berthold
Schweizer and Abe Sklar in \cite{schweizer1983b} gave an axiomatic for t-norms
as they are used today. Also, they play an important role in theories of fuzzy sets and fuzzy logic \cite{bede2013fuzzy} as they generalize the basic connectives between fuzzy sets. Thus, the main characteristic of the binary operations is that they are used in a large number of areas and disciplines. 

In \cite{de1999triangular,de1994order} they were generalized the notion of  t-norm on  bounded
partially ordered sets, which is a more general structure than
bounded lattice. In \cite{ray1997representation} it was considered an extension of t-norm
for bounded lattice which coincides with the one given
by \cite{de1994order} and \cite{de1999triangular}.

In 1970, E. Fried \cite{fried1970tournaments} introduced the notions of pseudo-ordered sets and
trellises (also called weakly associative lattices or WA-lattices). Trellises are generalization of lattices by considering sets with a reflexive and antisymmetric order, but not necessarily transitive. In 1971, H. X. Skala
\cite{Skala1971trellis} investigated some properties of this notion and provide some examples in particular classes of trellis and the notion of the trellis itself can be interpreted in terms of binary operations on it (see, e.g., \cite{fried1973some,yettou1}).

 The aim of the present paper is to introduce the notion of a pseudo-triangular norm (pseudo-t-norm, for short) and a pseudo-triangular conorm (pseudo-t-conorm, for short) on a bounded trellis as a generalization of triangular norm and  triangular conorm on a bounded lattice and bounded trellis (see, e.g., \cite{de1999triangular,de1994order,ertuugrul2015modified,klement2005logical,bedregal2006t,halavs2016clone}), and we investigate their fundamental properties and some constructions of pseudo-t-norms (resp.\ pseudo-t-conorms). More specifically, we show necessary and sufficient conditions under which a given binary operation on a
trellis coincides with its meet- and its join-operation. Mourover, we  characterize pseudo-t-norms (resp.\ pseudo-t-conorms) on  bounded trellises with respect to the $F$-distributivity.  Furthermore, we study the relationship among pseudo-t-norms and  isomorphisms on a bounded trellises.

This paper is organized as follows. We briefly recall some basic concepts in Section 2.
 In section 3, we introduce the notion of  pseudo-t-norm and pseudo-t-conorm on  bounded trellises and investigate their properties. In section 4, we construct some elements, some constructions and showing necessary and sufficient conditions under which a given binary operation (resp.\ pseudo-t-norm and pseudo-t-conorm) on a trellis (resp.\ bounded trellis) coincide with its meet- and its join-operation. In section 5, we introduce the notion of T-distributivity for any pseudo-t-norm (resp.\ pseudo-t-conorm) on  bounded trellises and characterize some properties.  In section 6, we show that any  isomorphism act on pseudo-t-norms generating also a new pseudo-t-norms on  bounded trellises. Finally, we present some conclusions and discuss future research in Section 7.

\section{Basic concepts}\label{Basic_concepts}
This section serves an introductory purpose. First, we recall some definitions and properties related to pseudo-ordered sets and trellises. Second, we present some specific elements of a trellis that will be needed throughout this paper.
 
\subsection{Pseudo-ordered sets and trellises}
In this subsection, we recall the notions of pseudo-ordered sets and trellises; more information can be found in~\cite{fried1970tournaments,Skala1971trellis,Skala1972trellis}.
A \textit{pseudo-order (relation)} $\unlhd$ on a set $X$ is a binary relation on $X$ that is reflexive (i.e., $x \unlhd x$, 
for any $x \in X$) and antisymmetric (i.e., $x \unlhd y$ and $y \unlhd x$ implies $x = y$, for any $x,y \in X$). A set $X$ equipped
with a pseudo-order relation $\unlhd$ is called a \textit{pseudo-ordered set} (\textit{psoset}, for short) and denoted by $(X,\unlhd)$. 
For any two elements $a, b \in X$, if $a \unlhd b$ and $a\neq b$, then we write $a \lhd b$; if $a \unlhd b$ does not hold, then we also 
write $a \ntrianglelefteq b$. Similarly as for partially ordered sets, a finite pseudo-ordered set can be represented by 
a \textit{Hasse-type diagram} with the following difference: if $x$ and $y$ are not related, while in a partially ordered
set this would be implied by transitivity, then $x$ and $y$ are joined by a dashed curve.

\begin{example}\label{example psoset} 
Consider the pseudo-order relation $\unlhd$ on $X=\{a,b,c,d,e,f\}$ represented by the Hasse-type diagram in Fig.~\ref{Fig01}. Here, $b\unlhd d$, $c \unlhd d$, $d\unlhd e$,
while $b\ntrianglelefteq e$ and $c \unlhd e$.

\begin{figure}[h]
\[\begin{tikzpicture}
\tikzstyle{estun}=[>=latex,thick,dotted]
    \vertex[fill] (a) at (0,0)  [label=right:$a$]  {};
    \vertex[fill] (b) at (-0.75,0.75)  [label=left:$b$]  {};
    \vertex[fill] (c) at (0.75,0.75)  [label=right:$c$]  {};
    \vertex[fill] (d) at (0,1.5)  [label=right:$d$]  {};
    \vertex[fill] (e) at (0,2.25)  [label=right:$e$]  {};
    \vertex[fill] (f) at (0,3)  [label=right:$f$]  {};
   
    \path
        (a) edge (b)
        (a) edge (c)
        (b) edge (d)
        (c) edge (d)
        (d) edge (e)
        (e) edge (f)       
        ;
   \draw[estun] (b) to [bend left] (e)   
         ;
\end{tikzpicture}\]
\caption{Hasse-type diagram of $(X,\unlhd)$.} \label{Fig01}
\end{figure}
\end{example}

The notions of \textit{minimal/maximal} element,  \textit{smallest/greatest} element, \textit{lower/upper bound}, 
\textit{greatest lower bound} (or \textit{infimum}), 
\textit{least upper bound} (or \textit{supremum}) for psosets are defined in the same way as the corresponding notions
for partially ordered sets. For a subset $A$ of a psoset $(X,\unlhd)$, the antisymmetry of the pseudo-order implies that if 
$A$ has an infimum (resp.\ supremum), then it is unique, and is denoted by $\bigwedge A$ (resp.\ $\bigvee A$). If $A=\{a, b\}$, 
then we  write $a\meet b$ (called \textit{meet}) instead of $\bigwedge \{a, b\}$ and $a\join b$ (called \textit{join}) instead of $\bigvee \{a, b\}$.

\begin{definition}\cite{Skala1972trellis}
Let $(X,\unlhd)$ be a psoset. For $x,y\in X$, we write $x\lesssim y$ if there exists a finite sequence $(x_{1}, \ldots, x_{n})$ such that 
$x \unlhd x_{1} \unlhd \ldots\unlhd x_{n} \unlhd y$.
\end{definition}

Note that the relation $\lesssim$ is a pre-order relation, i.e., it is reflexive and transitive, but not necessarily antisymmetric.
If for any $x,y \in X$, it holds that $x \lesssim y $ or $y \lesssim x $, then $(X,\unlhd)$ is called a \textit{pseudo-chain}.

\begin{definition}\cite{gladstien1973characterization}
A $\meet$-semi-trellis (resp.\ $\join$-semi-trellis) is a psoset $(X,\unlhd)$ such that $x \meet y$
(resp.~$x \join y$) exists for all $x,y \in X$.
\end{definition}

\begin{definition}\cite{Skala1971trellis}
A \textit{trellis} is a psoset that is both a $\meet$-semi-trellis and a $\join$-semi-trellis. In other
words, a trellis is an algebra $(X, \meet, \join)$, where $X$ is a nonempty set with the binary operations $\meet$ and $\join$ satisfying the following properties,
for any $a,b,c\in X$:
\begin{enumerate}[label=(\roman*)]
\item $a \join b = b \join a$ and $a \meet b = b \meet a$ (\textit{commutativity})\,;
\item $a \join( b \meet a) = a= a \meet (b \join a)$ (\textit{absorption})\,;
\item $a \join( (a \meet b) \join ( a \meet c)) = a= a \meet( (a \join b) \meet ( a \join c)) $ (\textit{part-preservation})\,.
\end{enumerate}
\end{definition}

\begin{theorem}{\rm \cite{Skala1971trellis}}
A set $X$ with two commutative, absorptive, and part-preserving operations $\meet$ and $\join$  is a trellis if $a \unlhd b$ is defined as $a \meet b=a$ and/or $a \join b=b$. The operations are also idempotent (i.e., $x \meet x= x \join x=x$, for any $x \in X$).
\end{theorem}

\begin{remark}\label{}
One can observe that the difference between the notions of a lattice and a trellis is that the operations $\meet$ and $\join$ are not required to be 
associative in the case of a trellis.
\end{remark}

A \textit{bounded trellis} is a trellis $(X, \unlhd, \meet, \join)$ that additionally has a smallest element denoted
by $0$ and a greatest element denoted by $1$ satisfying $0 \unlhd x \unlhd 1$, for any $x \in X$. For a bounded trellis,
the notation $(X, \unlhd, \meet, \join,0,1)$ is used. Also, a trellis $(X, \unlhd, \meet, \join)$ is called \textit{complete} 
if every subset of $X$ has an infimum and a supremum.

Let $(X, \unlhd, \meet, \join)$ and $(Y, \sqsubseteq, \sqcap, \sqcup)$ be two trellises. A mapping $\varphi: X\to Y$ is called a \textit{homomorphism},  if it satisfies  $\varphi(x\meet y)=\varphi(x) \sqcap \varphi(y)$ and $\varphi(x\join y)=\varphi(x)\sqcup \varphi(y)$, for any $x,y\in X$. An \textit{isomorphism} is a bijective homomorphism. 

\begin{definition} \cite{gladstien1973characterization}
Let $(X,\unlhd,\meet,\join)$ be a trellis and $A \subseteq X$. Then
\begin{enumerate}[label=(\roman*)]
\item $A$ is called a \textit{sub-trellis} of $X$ if $x \meet y \in A$ and $x \join y \in A$, for any $x,y \in A$;
\item $A$ is called a \textit{sub-lattice} of $X$ if is a sub-trellis and $\unlhd$ is transitive on $A$.
\end{enumerate}
\end{definition}

\begin{theorem}{\rm \cite{Skala1972trellis}}\label{the equivalent between the meet (resp.join) and the pseudo-order}
Let $(X,\unlhd,\meet,\join)$ be a trellis. The following statements are equivalent:
\begin{enumerate}[label=(\roman*)]  
\item $\unlhd$ is transitive;
\item the meet $\meet$ and the join $\join$ operations are associative; 
\item one of the operations meet $\meet$ or join $\join$ is associative.
\end{enumerate}
\end{theorem}

\begin{definition}\cite{Skala1972trellis}\label{Modular_Trellis}
A trellis $(X,\unlhd,\meet,\join)$ is said to be \textit{modular}, if $x \unlhd z$ implies that $x \join (y \meet z)= (x \join y ) \meet z$, for any $y \in X$. \end{definition}
We will also use the following results.

\begin{proposition}{\rm \cite{Skala1971trellis}}\label{two famous implications on modular trellis}
Let $(X,\unlhd,\meet,\join)$ be a modular trellis and $x,y,z \in X$. If $x \unlhd y \unlhd z$, then  $x \meet z \unlhd y \unlhd x \join z$. 
\end{proposition}

\begin{proposition}{\rm \cite{Skala1971trellis}}\label{Interesting_implication_modular_trellis}
Let $(X,\unlhd,\meet,\join,0,1)$ be a bounded modular trellis. If $x \unlhd z$ and $x \join y=1$, then $x\meet y \unlhd z$, for any  $x,y,z \in X$.
\end{proposition}

 \subsection{Specific elements in a trellis}
In this subsection, we present some specific elements in a trellis that will play an important role in this paper.

\begin{definition}\cite{Skala1972trellis}
Let $(X,\unlhd,\meet,\join)$ be a trellis. An element $\alpha \in X$ is called:
\begin{enumerate}[label=(\roman*)]  
\item \textit{right-transitive}, if $\alpha \unlhd x \unlhd y$ implies  $\alpha \unlhd y$, for any $x,y \in X$; 
\item \textit{left-transitive}, if $x \unlhd y \unlhd \alpha$ implies $x \unlhd \alpha$, for any $x,y \in X$;
\item \textit{middle-transitive}, if $x \unlhd \alpha \unlhd y$ implies $x \unlhd y$, for any $x,y \in X$;
\item \textit{transitive}, if it is right-, left- and middle-transitive.
\end{enumerate}
\end{definition}

\begin{definition}\cite{Skala1972trellis}
Let $(X,\unlhd,\meet,\join)$ be a trellis. 
\begin{enumerate}[label=(\roman*)]  
\item A 3-tuple $(x,y,z) \in X^3$ is called  $\meet$-\textit{associative} (resp.\ $\join$-\textit{associative}), if $(x \meet y )\meet z = x\meet(y \meet z)$ (resp.\ $(x\join y)\join z = x\join(y \join z)$);
\item An element  $\alpha \in X$ is called $\meet$-associative (resp.\ $\join$-associative), if any 3-tuple in $X$ including $\alpha$ is $\meet$-associative (resp.\ $\join$-associative); 
 \item An element $\alpha\in X$ is called \textit{associative} if it is both $\meet$- and $\join$-associative. 
\end{enumerate}
\end{definition}

Note that for the different notions of associative element $\alpha \in X$, due to the commutativity of the meet and the join operations it is 
sufficient to consider only 3-tuples if the type $(\alpha,x,y)$.
 
The following results show the links between the above notions. 

\begin{proposition}{\rm \cite{Skala1972trellis}}\label{associative element is transitive}
Let $(X,\unlhd,\meet,\join)$ be a trellis. Any $\meet$-associative or $\join$-associative element is transitive, but the converse does not hold.
\end{proposition}

\begin{theorem}{\rm \cite{Skala1972trellis}}\label{Ass=tr}
Let $(X,\unlhd,\meet,\join)$ be a pseudo-chain or a modular trellis. Then it holds that an element is associative if and only if it is transitive. 
\end{theorem}

\begin{theorem}\cite{gladstien1973characterization}\label{Every cycle has a meet and join}
A trellis of finite length is complete if and only if every cycle has the meet and the join.
\end{theorem}

\section{Pseudo-triangular norms on  bounded trellises}
This section is devoted to introduce the notions of pseudo-triangular norm  on a bounded trellis and to investigate their various properties and to present some interesting examples in bounded trellises. These notions are inspired from triangular norms and triangular conorms on  bounded lattices and bounded trellises (see, e.g., \cite{de1999triangular,de1994order,karacal2017aggregation,klement2005logical}). Also, we provide a construction to obtain new pseudo-t-norms  on bounded trellises. In particular, we give necessary and sufficient conditions under which a pseudo-t-norm on a bounded trellis coincides with its meet ($\wedge$) operation.

\subsection{Binary operations on trellises}
In this subsection, we present some basic definitions and properties of binary operations on a psoset or trellis. 
Some of them are adopted from the corresponding notions on a poset or lattice (see, e.g.,~\cite{fried1973some,karaccal2006direct,yettou1}).
A binary operation $F$ on a psoset $(X,\unlhd)$
is called:
\begin{enumerate}[label=(\roman*)]
\item \textit{commutative}, if $F(x,y)=F(y,x)$, for any $x,y \in X$;
\item \textit{associative}, if $F(x,F(y,z))=F(F(x,y),z)$, for any $x,y,z \in X$;
\item \textit{idempotent}, if $F(x,x)=x$, for any $x \in X$;
\item \textit{increasing}, if $x \unlhd y$  implies   $F(x,z) \unlhd F(y,z)$, for any $z \in X$.
\end{enumerate} 

 An element $e\in X$ is called a \textit{neutral element} of $F$, if $F(e,x)=F(x,e)=x$, for any $x \in X$.

A binary operation $F$ on a trellis $(X,\unlhd,\meet ,\join)$
is called:
\begin{enumerate}[label=(\roman*)]
\item \textit{conjunctive}, if $F(x,y)\unlhd x\meet y$, for any $x,y \in X$;
\item \textit{disjunctive}, if $x\join y \unlhd F(x,y)$, for any $x,y \in X$.
\end{enumerate}

\begin{remark}
Consider a trellis $(X,\unlhd ,\meet ,\join )$.Then the meet $\meet$ (resp.\ join $\join$) is conjunctive (resp.\  disjunctive). 
\end{remark}

 \begin{notation}
Let $(X,\unlhd,\meet,\join)$ be a trellis. We denote  by:
\begin{enumerate}[label=(\roman*)] 
\item $X^{tr}$ : the set of all transitive elements of $X$;
\item $X^{\wedge-ass}$:  the set of all $\meet$-associative elements of $X$;
\item  $X^{\vee-ass}$:  the set of all $\join$-associative elements of $X$;
\item  $X^{ass}$: the set of all associative elements of $X$;
\item  $X^{dis}$:  the set of all distributive elements of $X$.
\end{enumerate}  
\end{notation}

\begin{notation}\label{}
Let  $(X,\unlhd,\wedge,\vee)$ be a trellis, $A \in X$ and $x_1,\cdots , x_n \in X$, for any $n \geq 1$. If $\{x_1,\cdots , x_n\} \cap A \neq \emptyset $, then we said that  $[x_1,\cdots , x_n ] \in A$.
\end{notation}

The following proposition is immediate.

\begin{proposition}\label{the meet and the join are p-increasing}
Let $(X,\unlhd ,\wedge ,\vee )$ be a trellis. Then it holds that 
\begin{enumerate}[label=(\roman*)]
\item $x \unlhd y \text{ implies }  x \wedge z \unlhd y \wedge z \text{ and } z \wedge x \unlhd z \wedge y, \text{ for any } ([x,y] \in X^\tr \text{ and }  z \in X);$
\item $x \unlhd y \text{ implies }  x \vee z \unlhd y \vee z \text{ and } z \vee x \unlhd z \vee y, \text{ for any } ([x,y] \in X^\tr \text{ and }  z \in X);$
\item $(x \wedge  y )\wedge  z= x \wedge ( y \wedge z ) \text{, for any } ([x,y,z ] \in X^{\wedge-ass} \text{ or } [x,y,z ]\in X^{\wedge-ass})\ ;$
\item $(x \vee  y )\vee  z= x \vee ( y \vee z ) \text{, for any } ([x,y,z ] \in X^{\wedge-ass} \text{ or } [x,y,z ]\in X^{\vee-ass})\ .$
\end{enumerate}
\end{proposition}

Next, we extends the increasingness and associativity properties of  the meet and the join  operations on bounded trellises and leads the following definition. 

\begin{definition}
Let  $(X,\unlhd,\wedge,\vee)$ be a trellis and $F$  a binary operation on $X$.
\begin{enumerate}[label=(\roman*)]  
\item $F$ is called weakly-increasing if it satisfies: 
$$ x \unlhd y \Rightarrow  F(x,z) \unlhd F(y,z), \text{ for any } ([x,y] \in X^\tr \text{ and }  z \in X);$$
\item $F$ is weakly-associative if it satisfies: 
$$F(x,F(y,z)) = F(F(x,y),z)  \text{, for any } ([x,y,z ] \in X^{\wedge-ass} \text{ or } [x,y,z ]\in X^{\vee-ass})\ .$$
\end{enumerate}
\end{definition}

Next, we illustrate the previous definition weakly-increasing and weakly-associative    operations on a bounded trellis.

\begin{example}
Let  $(X=\{0,a,b,c,1\}, \unlhd, \wedge , \vee)$ be a trellis given by the Hasse  diagram in Figure~\ref{Fig4.1} and $F,G$  two binary operations defined by the following tables:
\begin{center}
\begin{tabular}{|c|c|c|c|c|c|}
\hline 
$F(x,y)$ & $0$ & $a$ & $b$ & $c$ & $1$ \\ 
\hline 
$0$ & $a$ & $a$ & $b$ & $c$ & $1$ \\ 
\hline 
$a$ & $b$ & $b$ & $c$ & $c$ & $1$ \\ 
\hline 
$b$ & $b$ & $b$ & $c$ & $c$ & $1$ \\ 
\hline 
$c$ & $c$ & $1$ & $1$ & $1$ & $1$ \\ 
\hline 
$1$ & $1$ & $1$ & $1$ & $1$ & $1$ \\ 
\hline 
\end{tabular}$~~~~$
\begin{tabular}{|c|c|c|c|c|c|}
\hline 
$G(x,y)$ & $0$ & $a$ & $b$ & $c$ & $1$ \\ 
\hline 
$0$      & $0$ & $0$ & $0$ & $0$ & $0$ \\ 
\hline 
$a$ 	 & $0$ & $a$ & $b$ & $c$ & $1$ \\ 
\hline 
$b$		 & $0$ & $a$ & $c$ & $c$ & $1$ \\ 
\hline 
$c$ 	 & $0$ & $0$ & $1$ & $b$ & $c$ \\ 
\hline 
$1$ 	 & $0$ & $a$ & $c$ & $c$ & $1$ \\ 
\hline 
\end{tabular}
\end{center}

\vspace{0.3cm} 

One  easily verifies that $F$ is weakly-increasing and $G$ is weakly-associative. 

\begin{figure}[H]

\[\begin{tikzpicture}
\tikzstyle{estun}=[>=latex,thick,dotted]
    \vertex[fill] (0) at (0,0)  [label=right:$0$]  {};
    \vertex[fill] (a) at (0,0.5)  [label=right:$a$]  {};
    \vertex[fill] (b) at (0,1)  [label=right:$b$]  {};
    \vertex[fill] (c) at (0,1.5)  [label=right:$c$]  {};
    \vertex[fill] (1) at (0,2)  [label=right:$1$]  {};
   
    \path
        (0) edge (a)
        (a) edge (b)
        (b) edge (c)
        (c) edge (1)
        
        ;
   \draw[estun] (a) to [bend left] (c)   
      
        ;
        
\end{tikzpicture}\]

\caption{Hasse  diagram of the trellis $(X=\lbrace 0,a,b,c,1 \rbrace,\unlhd)$.}\label{Fig4.1}
\end{figure}
\end{example}

\newpage
\subsection{Triangular norms  on  bounded trellises}
In this subsection, we introduce the notion of  triangular norm and on a bounded trellis and we present some illustrative examples.

\begin{definition}\label{t-norm}
Let $(X,\unlhd,\wedge,\vee,0,1)$ be a bounded trellis. A binary operation $T: X^2 \to X$ is called a triangular norm (t-norm, for short), if it is commutative, increasing, associative
and has $1$ as neutral element, i.e., $T(1, x)=x$, for any $x \in X$.
\end{definition}

Analogously, we define a triangular conorm on a bounded trellis.

\begin{definition}\label{t-conorm}
Let $(X,\unlhd,\wedge,\vee,0,1)$ be a bounded trellis. A binary operation $S: X^2 \to X$ is called a pseudo-triangular conorm (pseudo-t-conorm, for short), if it is commutative, increasing, associative and has $0$ as neutral element, i.e., $S(0, x)=x$, for any $x \in X$.
\end{definition}

\begin{example}{\ }\label{T_D and S_D}
Let $(X,\unlhd,\wedge,\vee,0,1)$ be a bounded trellis and $T_D$ a binary operation defined on $X$ as follow:\\
$T_D(x,y)=\left\{\begin{array}{ll}{x \wedge y} & {\text { if } x =1 ~\text{or}~ y=1},\\
 {0} & {\text { otherwise, }}\end{array}\right.$
and~
$S_D(x,y)=\left\{\begin{array}{ll}{x \vee y} & {\text { if } x =0 ~\text{or}~ y=0},\\
 {1} & {\text { otherwise. }}\end{array}\right.$

One  easily verifies that $T_D$  is the smallest t-norm and $S_D$ is the greatest t-conorm  on $X$.
\end{example}

\begin{example} 
 Let $(X=\{0,a,b,c,1\}, \unlhd, \wedge , \vee)$ be a bounded trellis given by the Hasse-type  diagram in Figure~\ref{Fig4.1}  and $T$  a binary operation defined by the following table:\\

\begin{center}
\begin{tabular}{|c|c|c|c|c|c|}
\hline 
$T(x,y)$ &$0$ & $a$& $b$ & $c$ & $1$ \\ 
\hline 
$0$    &$0$ & $0$& $0$ & $0$ & $0$ \\  
\hline  
$a$	   &$0$ & $0$& $0$ & $0$ & $a$ \\  
\hline 
$b$	   &$0$ & $0$& $b$ & $b$ & $b$ \\ 
\hline 
$c$    &$0$ & $0$& $b$ & $c$ & $c$ \\ 
\hline 
$1$    &$0$ & $a$& $b$ & $c$ & $1$ \\ 
\hline 
\end{tabular}$~~~~$
\end{center}
One easily verifies that $T$ is a t-norm on $X$ such that $T_D \vartriangleleft T$. 
\end{example} 

\begin{example}{\ }
 Let $(X,\unlhd,\wedge,\vee,0,1)$ be a bounded trellis such that for any $i \in$ $Coatom(X)$ and $j \in$ $Atom(X)$, the binary operations $T_i$ and $S_j$ defined as:
\begin{center}
$T_i(x,y)=\left\{\begin{array}{ll}{i} & {\text { if } (x,y)=(i,i)},\\
 {T_D(x,y)} & {\text { otherwise, }}\end{array}\right.$
\end{center}
 
is a t-norm on $X$. Moreover, $T_D \vartriangleleft T_i$.
\end{example}
\newpage
\begin{notation}
Let $(X,\unlhd ,\wedge ,\vee )$ be a bounded trellis. We denote by:
\begin{enumerate}[label=(\roman*)]
\item $\mathcal{AO}_1(X)$:  the class (or the set) of all t-norms on $ X$;
\item $\mathcal{AO}_0(X)$:  the class (or the set) of all t-conorms on $ X$.
\end{enumerate}
\end{notation}

\begin{remark}{\ }\label{existence idempotent t-norm on bounded trellis}
In a bounded trellis  $(X,\unlhd,\wedge,\vee,0,1)$ wish has not a structure of a bounded lattice, there is not exist any idempotent t-norm on $X$.
\end{remark}

It is natural that the trellis structure has cycles and  dashed curves at the same time. In the following results, we  present relationship among cycles and t-norms. First, we need the following definition of atom and coatom on a trellis. This definition is a natural generalization of the same notions  on a lattice (see, e.g.~\cite{karaccal2006direct}).
\begin{definition}
Let $(X,\unlhd,\wedge,\vee,0,1)$ be a bounded trellis. An element $\alpha \in X$ is called:
\begin{enumerate}[label=(\roman*)]  
\item \textit{atom},  if it is a minimal element of the set $X \backslash\{0\}$;
\item \textit{coatom}, if it is a maximal element of the set $X \backslash\{1\}$.
\end{enumerate}
We denote by $Atom(X)$ (resp.\ $Coatom(X)$), the set of all atoms (resp.\ coatoms) of $X$.
\end{definition}

The following propositions are immediate.

\begin{proposition}
For a given bounded trellis $(X,\unlhd,\wedge,\vee,0,1)$ has a non-trivial cycle $C$. Then $\{0,1\} \notin C$.
\end{proposition}

\begin{proposition}
Let $(X,\unlhd,\wedge,\vee,0,1)$ be a bounded trellis. If $X$ has a non-trivial cycle $C$, then $x \notin Atom(X) \cup Coatom(X)$, for any $x \in C$.
\end{proposition}

\begin{proposition}
Let $(X,\unlhd,\wedge,\vee)$ be a modular trellis. Then $X$ does not have any non-trivial cycle.
\end{proposition}
\begin{proof}{\ }
Let $C=\{x,y,z\}$ be a non-trivial cycle on $X$ such that  $x \unlhd y \unlhd z \unlhd x$. Since $z \unlhd x$, it holds from the modularity of $X$ that $(z \vee y) \wedge x=z \wedge x= z$ and $z \vee (y \wedge x)=z \vee x= x$. Then $x=z$.  Consequently, $x=y=z$, contradiction. Hence, $X$ does not have any non-trivial cycle.
\end{proof}

\begin{proposition}
Let $(X,\unlhd,\wedge,\vee)$ be a trellis has a non-trivial cycle $C$. Then it holds that  
\begin{center}
$C \cap \{ X^{r-tr} \cup  X^{\ell -tr} \} = \emptyset$.
\end{center}
\end{proposition}
\begin{proof}{\ }
Let $C=\{ x_{1}, \cdots, x_{n}\}$ be a non-trivial cycle such that $x_{1} \unlhd x_2 \unlhd \cdots \unlhd  x_{n} \unlhd x_1$ and  $n \geq 3$. Suppose that $x_i \in X^{r-tr}$, for some $i \in \{1, \cdots , n\}$. Then $ x_i \unlhd x_{i-1}$. Since $ x_{i-1} \unlhd x_i$, it holds that $x_{i}= x_{i-1}$. Consequently, $x_{i}= x_{i-1}= x_{i-2}=\cdots =x_1=x_n=\cdots = x_{i+1}$. Thus, $|C|=1$, a contradiction. Hence, $C \cap X^{r-tr}  = \emptyset$. In similar way, $C \cap X^{\ell -tr}  = \emptyset$. Therefore, $C \cap \{ X^{r-tr} \cup  X^{\ell -tr} \} = \emptyset$.
\end{proof}

\begin{proposition}\label{the set of all  non-trivial cycles not coincide with a t-norm}
Let $(X,\unlhd,\wedge,\vee)$ be a trellis and $\{ C_i \}_{i \in I}$ is the set of all  non-trivial cycles of three elements of $X$. If $T$ is a t-norm on $X$, then $T(x,y) \notin C_i$, for any $x,y \in C_i$ and $i\in I$ and . Moreover,  $T(x,y)=T(z,t)$, for any $x,y,z,t \in C_i$ and $i\in I$ (i.e., $\bigcup_{a,b \in C_i} T(a,b)$ is a trivial cycle).
\end{proposition}

\begin{proof}{\ }
 Let  $C=\{x,y,z\} \subseteq \bigcup_{i \in I} C_i$ be a non-trivial cycle such that $x \unlhd y \unlhd z \unlhd x$. Suppose that $a,b \in C$ such that $T(a,b)=x$. Since $T$ is a t-norm, it holds that $T(a,b)= x \unlhd a $ and  $T(a,b)= x \unlhd b $. Thus, $a \in \{x, y\}$ and $b \in \{x, y\}$. Then the commutativity of $T$  implies that two cases:
 \begin{enumerate}[label=(\roman*)]  
\item If $a=x$ and $b=y$, then $T(a,b)=T(x,y)=x$. Since $T(x,y) \unlhd T(y,z)$, it follows two possible cases:
 \begin{enumerate}[label=(\roman*)]  
\item If $T(y,z)=x$, then $T(y,z)=x \unlhd T(1,z)=z$. Thus, $x \unlhd z$. Hence, $x=y=z$.
\item If $T(y,z)=y$, then $T(y,z) \unlhd T(y,x)=T(x,y)=x$. Thus, $y \unlhd x$. Hence, $x=y=z$.
\end{enumerate}
\item If $a=b=x$, then $T(a,b)=T(x,x)=x$. Since $T(x,x) \unlhd T(y,y)$, it holds that two cases:
\begin{enumerate}[label=(\roman*)]  
\item If $T(y,y)=x$, then $T(x,x)= T(y,y)=x$. Since $T(x,x) \unlhd T(x,y) \unlhd T(y,y)$, it holds that $T(x,y)=x$, a contradiction. 
\item If $T(y,y)=y$, Since $T(x,x) \unlhd T(x,y) \unlhd T(y,y)$, it holds that $T(x,y)=x$ or $T(x,y)=y$. Thus,  $T(x,y)=x$ is a contradiction and $T(x,y)=y \unlhd T(x,1)=x$. Hence, $y \unlhd x$. Hence, $x=y=z$.
\end{enumerate}
\item If $a=b=y$, then $T(a,b)=T(y,y)=x$. Since $T(x,x) \unlhd T(y,y)$, it holds that $T(x,x)=x$ or $T(x,x)=z$. Thus,  $T(x,x)=x$ is a contradiction and $T(x,x)=z \unlhd T(x,y) \unlhd T(y,y)=x$ implies that $T(x,y)=x$ or  $T(x,y)=z$, a contradiction.
\end{enumerate}
In similar way, if $T(a,b)=y$ or $T(a,b)=z$, it follows that $ x=y=z$. Therefore,  $T(a,b) \notin C$, for any $a,b \in C$.  Next, let  $C'= \{\alpha , \beta , \gamma \}$ an other non-trivial cycle such that $| \bigcup_{a,b \in C} T(a,b) \cap C^{'} | \geq 2$. Suppose that  $T(x,x)=\alpha$, $T(y,y)=\beta$. Then  $T(z,z)=\gamma$. The fact that  $T$ is a t-norm implies  $T(x,x) \unlhd T(x,y) \unlhd T(y,y)$. Then it holds that two cases:
\begin{enumerate}[label=(\roman*)]  
\item If $T(x,y)= \alpha$, then from $T$ is a t-norm, it holds that $T(x,y) \unlhd T(y,z) \unlhd T(z,z)$. Thus, $T(y,z)= \beta$ . The fact that $T(y,z)=T(z,y) \unlhd T(x,y)$ implies 
 $T(x,y) = T(y,z)$. Then $\alpha =\beta$. Hence, $\alpha =\beta= \gamma$. 
\item If $T(x,y)= \beta$, then from $T$ is a t-norm, it holds that $T(x,y)=\beta  \unlhd T(x,z) \unlhd T(x,x)=\alpha$. Thus,  $T(x,z)=\gamma$. Since $T(x,z)=T(z,x) \unlhd T(x,y)$, it holds that $\gamma \unlhd \beta$. Hence,  $\alpha =\beta= \gamma$.
\end{enumerate}
In similar way, if $T(a,b)=\alpha$ and $T(c,d)=\beta$, for any $a,b,c,d \in C$, it follows that $\alpha =\beta= \gamma$. Hence, $\bigcup_{a,b \in C} T(a,b)$ is a trivial cycle (i.e., $T(x,y)=T(z,t)$, for any $x,y,z,t \in C$. 
\end{proof}

\begin{proposition}\label{t-norms with a cycles}
Let $(X,\unlhd,\wedge,\vee)$ be a trellis and $ C_n$ is a non-trivial cycle of $n$ elements of $X$. If $T$ is a t-norm on $X$, then $T(x,y) \notin C_n$, for any $x,y \in C_n$. Moreover,  $T(x,y)=T(z,t)$, for any $x,y,z,t \in C_n$ (i.e., $\bigcup_{a,b \in C_n} T(a,b)$ is a trivial cycle).
\end{proposition}

\begin{proof}{\ }
Let $T$ is a t-norm on $X$ and $C_3$ is anon-trivial cycles of three elements, then  Proposition ~\ref{the set of all  non-trivial cycles not coincide with a t-norm} guarantees that  $T(x,y) \notin C_3$, for any $x,y \in C_3$ and  $T(x,y)=T(z,t)$, for any $x,y,z,t \in C_3$. Suppose that $C_n=\{ x_{1}, \cdots, x_{n}\}$ be a non-trivial cycle such that $T(x,y) \notin C_n$, for any $x,y \in C_n$ and  $T(x,y)=T(z,t)$, for any $x,y,z,t \in C_n$.  Next, we prove that for $C_{n+1}=\{ x_{1}, \cdots, x_{n} ,x_{n+1} \}$. First, let $T(x,y) \in C_{n+1}$, for any $x,y \in C_{n}$. Since $T(x,y) \notin C_{n}$, for any $x,y \in C_{n}$, it holds that $T(x,y)=x_{n+1}$. Suppose that $x=y=x_n$, then $T(x,y)=T(x_n,x_n)=x_{n+1}$. Since $T$ is a t-norm, it holds that $T(x_n,x_n)  \unlhd x_n$. Thus, $x_n \unlhd x_{n+1} = T(x_n,x_n) \unlhd x_n$. Hence, $T(x_n,x_n) = x_n \in C_n$, a contradiction. Thus, $T(x,y) \notin C_{n+1}$, for any $x,y \in C_{n}$. Second, we prove that $T(x,y)=T(x_{n+1},z)=T(x_{n+1},x_{n+1})$, for any $x,y,z \in C_n$. On the one hand, since $T$ is a t-norm, it  follows that $T(x_{n},x_{n}) \unlhd T(x_{n+1},x_{n+1}) \unlhd T(x_{1},x_{1})=T(x_{n},x_{n})$. Then $T(x_{n+1},x_{n+1})=T(x_{n},x_{n})$. Thus,  $T(x_{n+1},x_{n+1})=T(x,y)$, for any $x,y \in C_n$ (using our hypothesis). On the other hand, $T(x_{n},z) \unlhd T(x_{n+1},z)\unlhd T(x_{1},z)= T(x,y)$, for any $x,y,z \in C_n$. Thus, $T(x_{n+1},z)= T(x,y)$, for any $x,y,z \in C_n$. Hence, $T(x,y)=T(x_{n+1},z)=T(x_{n+1},x_{n+1})$, for any $x,y,z \in C_n$ (i.e., $T(x,y)=T(z,t)$, for any $x,y,z,t \in C_{n+1}$).  Since $T(x,y) \notin C_{n+1}$, for any $x,y \in C_{n}$, it follows that  $T(x_{n+1},z)=T(x_{n+1},x_{n+1})  \notin C_{n+1}$,  for any $z \in C_{n+1}$. Consequently, $T(x,y) \notin C_{n+1}$, for any $x,y \in C_{n+1}$. 
\end{proof}

Theorem~\ref{Every cycle has a meet and join} and Proposition~\ref{t-norms with a cycles} leads to the following corollary.

\begin{corollary}{\ }\label{}
Let $(X,\unlhd,\wedge,\vee,0,1)$ be a finite bounded trellis and complete. If $T$ is a t-norm on $X$  such that $\{  C_i \}_{i \in I}$ is the set of all non-trivial cycles on $X$ which contains $x$ and $y$, then it holds that  $T(x,y) \unlhd T(\bigwedge ( \bigcup_{i \in I} C_i) ,\bigwedge (\bigcup_{i \in I} C_i))$.
\end{corollary}

%

 In the following illustrative example, we give all t-norms on a given bounded trellis has one cycle.
 
\begin{example}
Let $(X=\{0,a,b,c,1\}, \unlhd, \wedge , \vee)$ be a bounded trellis given by the Hasse-type diagram in Figure~\ref{Fig3}. Then the only  t-norm on $X$ is $T_W$.
 
 \begin{figure}[H]
\[\begin{tikzpicture}
\tikzstyle{estun}=[>=latex,thick,dotted]
    \vertex[fill] (0) at (0,0)  [label=right:$0$]  {};
    \vertex[fill] (a) at (0,0.5)  [label=right:$a$]  {};
    \vertex[fill] (b) at (0,1)  [label=right:$b$]  {};
    \vertex[fill] (c) at (0,1.5)  [label=right:$c$]  {};
    \vertex[fill] (1) at (0,2)  [label=right:$1$]  {};
   
    \path
        (0) edge (a)
        (c) edge (1)
        ;
    \path[->]             
        (a) edge (b)
        (b) edge (c)

        ;
   \draw[->] (c) to [bend right=60] (a)   
      
        ;
       
\end{tikzpicture}\]
\caption{The Hasse-type  diagram of the trellis $(X=\lbrace 0,a,b,c,1 \rbrace,\unlhd)$.}\label{Fig3}
\end{figure}
\end{example} 

\begin{proposition}{\ }\label{pseudo-chain have at most two dashed curve}
The smallest pseudo-chain have at most two dashed curve, are isomorphic to one of the following pseudo-chains defined as follows:

 \begin{figure}[H]
\[\begin{tikzpicture}
\tikzstyle{estun}=[>=latex,thick,dotted]

    \vertex[fill] (0) at (-6,0)  [label=right:$0$]  {};
    \vertex[fill] (a) at (-6,0.75)  [label=right:$a$]  {};
    \vertex[fill] (b) at (-6,1.5)  [label=right:$b$]  {};
    \vertex[fill] (c) at (-6,2.25)  [label=right:$c$]  {};
    \vertex[fill] (1) at (-6,3)  [label=right:$1$]  {};
   
    \path
        (0) edge (a)
        (a) edge (b)
        (b) edge (c)
        (c) edge (1)

       ;
   \draw[estun] (a) to [bend left=60] (c) ;

    \vertex[fill] (0) at (-4,0)  [label=right:$0$]  {};
    \vertex[fill] (a) at (-4,0.75)  [label=right:$a$]  {};
    \vertex[fill] (b) at (-4,1.5)  [label=right:$b$]  {};
    \vertex[fill] (c) at (-4,2.25)  [label=right:$c$]  {};
    \vertex[fill] (d) at (-4,3)  [label=right:$d$]  {};
    \vertex[fill] (1) at (-4,3.75)  [label=right:$1$]  {};
   
    \path
        (0) edge (a)
        (a) edge (b)
        (b) edge (c)
        (c) edge (d)
        (d) edge (1)  
        
       ;
   \draw[estun] (a) to [bend left=60] (c) ;         
      \draw[estun] (a) to [bend left=70] (d) ;

    \vertex[fill] (0) at (-2,0)  [label=right:$0$]  {};
    \vertex[fill] (a) at (-2,0.75)  [label=right:$a$]  {};
    \vertex[fill] (b) at (-2,1.5)  [label=right:$b$]  {};
    \vertex[fill] (c) at (-2,2.25)  [label=right:$c$]  {};
    \vertex[fill] (d) at (-2,3)  [label=right:$d$]  {};
    \vertex[fill] (1) at (-2,3.75)  [label=right:$1$]  {};
   
    \path
        (0) edge (a)
        (a) edge (b)
        (b) edge (c)
        (c) edge (d)
        (d) edge (1)

       ;
   \draw[estun] (a) to [bend left=60] (c) ;         
      \draw[estun] (b) to [bend left=60] (d) ;

    \vertex[fill] (0) at (2,0)  [label=right:$0$]  {};
    \vertex[fill] (a) at (2,0.75)  [label=right:$a$]  {};
    \vertex[fill] (b) at (2,1.5)  [label=right:$b$]  {};
    \vertex[fill] (c) at (2,2.25)  [label=right:$c$]  {};
    \vertex[fill] (d) at (2,3)  [label=right:$d$]  {};
    \vertex[fill] (e) at (2,3.75)  [label=right:$e$]  {};
    \vertex[fill] (1) at (2,4.5)  [label=right:$1$]  {};
   
    \path
        (0) edge (a)
        (a) edge (b)
        (b) edge (c)
        (c) edge (d)
        (d) edge (e)  
        (e) edge (1)
        
       ;
   \draw[estun] (a) to [bend left=70] (c) ;         
      \draw[estun] (c) to [bend left=60] (e) ; 
      
    \vertex[fill] (0) at (0,0)  [label=right:$0$]  {};
    \vertex[fill] (a) at (0,0.75)  [label=right:$a$]  {};
    \vertex[fill] (b) at (0,1.5)  [label=right:$b$]  {};
    \vertex[fill] (c) at (0,2.25)  [label=right:$c$]  {};
    \vertex[fill] (d) at (0,3)  [label=right:$d$]  {};
    \vertex[fill] (1) at (0,3.75)  [label=right:$1$]  {};
   
    \path
        (0) edge (a)
        (a) edge (b)
        (b) edge (c)
        (c) edge (d)
        (d) edge (1)

       ;
   \draw[estun] (a) to [bend left=70] (d) ;         
      \draw[estun] (b) to [bend left=60] (d) ; 
      
          \vertex[fill] (0) at (4,0)  [label=right:$0$]  {};
    \vertex[fill] (a) at (4,0.75)  [label=right:$a$]  {};
    \vertex[fill] (b) at (4,1.5)  [label=right:$b$]  {};
    \vertex[fill] (c) at (4,2.25)  [label=right:$c$]  {};
    \vertex[fill] (d) at (4,3)  [label=right:$d$]  {};
    \vertex[fill] (e) at (4,3.75)  [label=right:$e$]  {};
    \vertex[fill] (1) at (4,4.5)  [label=right:$1$]  {};
   
    \path
        (0) edge (a)
        (a) edge (b)
        (b) edge (c)
        (c) edge (d)
        (d) edge (e)  
        (e) edge (1)
        
       ;
   \draw[estun] (a) to [bend left=70] (e) ;         
      \draw[estun] (b) to [bend left=60] (d) ; 
       
\end{tikzpicture}\]
\end{figure}
\end{proposition}

\begin{remark} 
The bounded pseudo-chains have at most two dashed curve defined in Proposition~\ref{pseudo-chain have at most two dashed curve} has the greatest t-norm (using Matlab Program). In general, the greatest t-norm (resp.\ the smallest t-conorm)  on a arbitrary bounded pseudo-chain does not necessarily exist. Indeed, let $(X=\{0,a,b,c,d,e,f,1\}, \unlhd , \wedge , \vee)$ be a bounded pseudo-chain have three dashed curve given by the Hasse-type  diagram in Figure~\ref{Fig5}  and $T_{1}$ a binary operation on $X$ defined by the following table:
\begin{center}
\begin{tabular}{|c|c|c|c|c|c|c|c|c|}
\hline 
$T_{1}$ & $0$ & $a$ & $b$ & $c$ & $d$ & $e$  & $f$ & $1$ \\ 
\hline 
$0$ & $0$ & $0$ & $0$ & $0$ & $0$ & $0$  & $0$ & $0$   \\ 
\hline 
$a$ & $0$ & $0$ & $0$ & $0$ & $0$ & $0$  & $0$ & $a$ \\ 
\hline 
 $b$ & $0$ & $0$ & $b$ & $b$ & $b$ & $b$ & $b$ & $b$ \\ 
\hline 
$c$ & $0$ & $0$ & $b$ & $c$ & $c$ & $c$ & $c$ & $c$  \\ 
\hline 
$d$ & $0$ & $0$ & $b$ & $c$ & $c$ & $c$  & $c$ & $d$  \\ 
\hline 
$e$ & $0$ & $0$ & $b$ & $c$ & $c$ & $e$   & $e$ & $e$  \\ 
\hline 
$f$& $0$ & $0$ & $b$ & $c$ & $c$ & $e$  & $f$ & $f$  \\ 
\hline 
$1$ & $0$ & $a$ & $b$ & $c$ & $d$ & $e$  & $f$ &  $1$  \\ 
\hline
\end{tabular}$~~~~$
\end{center}

One easily verifies that $T_{1}$  is a maximal t-norm on $X$ (using Matlab Program). On other hand, for guarantees that $T_{1}$ is not the greatest t-norm on $X$, it is enough to find one t-norm  on $X$ such that $T \ntrianglelefteq T_{1}$. Let $T_{2}$  be a binary operation on $X$ defined by the following table:
\begin{center}
\begin{tabular}{|c|c|c|c|c|c|c|c|c|}
\hline 
$T_{2}$ & $0$ & $a$ & $b$ & $c$ & $d$ & $e$ & $f$ & $1$ \\ 
\hline 
$0$ & $0$ & $0$ & $0$ & $0$ & $0$ & $0$ & $0$ & $0$   \\ 
\hline 
$a$ & $0$ & $0$ & $0$ & $0$ & $0$ & $0$  & $0$  & $a$ \\ 
\hline 
 $b$ & $0$ & $0$ & $0$ & $0$ & $0$ & $b$  & $b$  & $b$ \\ 
\hline 
$c$ & $0$ & $0$ & $0$ & $0$ & $0$ & $b$  & $b$  & $c$  \\ 
\hline 
$d$ & $0$ & $0$ & $0$ & $0$ & $a$ & $b$  & $c$  & $d$  \\ 
\hline 
$e$ & $0$ & $0$ & $b$ & $b$ & $b$ & $e$  & $e$  & $e$  \\ 
\hline 
$f$& $0$ & $0$ & $b$ & $b$ & $c$ & $e$  & $e$  & $f$  \\  
\hline 
$1$ & $0$ & $a$ & $b$ & $c$ & $d$ & $e$ & $f$ &  $1$  \\ 
\hline
\end{tabular}$~~~~$
\end{center}
 One easily verifies that $T_{2}$ is a t-norm on $X$ (Matlab program). Let $x,y \in X$ such that $x=y=d$. Then $T_{1}(d,d)=c$ and $T_{2}(d,d)=a$. Since $a$ and $c$ are incomparable, it holds that $T_{1}$ and  $T_{2}$ are incomparable t-norms. Hence, $T_{1}$ is not the greatest t-norm on $X$.
 
 \begin{figure}[H]
\[\begin{tikzpicture}
\tikzstyle{estun}=[>=latex,thick,dotted]
    \vertex[fill] (0) at (0,0)  [label=right:$0$]  {};
    \vertex[fill] (a) at (0,0.75)  [label=right:$a$]  {};
    \vertex[fill] (b) at (0,1.5)  [label=right:$b$]  {};
    \vertex[fill] (c) at (0,2.25)  [label=right:$c$]  {};
    \vertex[fill] (d) at (0,3)  [label=right:$d$]  {};
    \vertex[fill] (e) at (0,3.75)  [label=right:$e$]  {};
    \vertex[fill] (f) at (0,4.5)  [label=right:$f$]  {};
    \vertex[fill] (1) at (0,5.25)  [label=right:$1$]  {};
   
    \path
        (0) edge (a)
        (a) edge (b)
        (b) edge (c)
        (c) edge (d)
        (d) edge (e)
        (e) edge (f)   
        (f) edge (1)
        
       ;
   \draw[estun] (a) to [bend left=70] (c) ;         
      \draw[estun] (a) to [bend left=60] (f) ;   
   \draw[estun] (d) to  [bend left=70] (f)
          
        ;
       
\end{tikzpicture}\]
\caption{The Hasse-type  diagram of the pseudo-chain $(X=\{0,a,b,c,d,e,f,1\},\unlhd)$.}\label{Fig5}
\end{figure}
Moreover, by a computer program (Matlab program),  the number of all t-norms is 159.
\end{remark} 

\subsection{Pseudo-triangular norms  on bounded trellises}
In this subsection, we introduce the notions of pseudo-triangular norms  on bounded trellises  and present some illustrative examples.

\begin{problem} Remark~\ref{existence idempotent t-norm on bounded trellis} leads  to research a new pseudo  triangular norms  on  bounded trellises as t-norms on  bounded lattices including the meet operation of  bounded trellises. In particular, an idempotent pseudo  triangular norms on  bounded trellises. Moreover,  the pseudo  triangular norms on bounded trellises  extends   triangular norms on bounded lattices.
\end{problem}

\begin{definition}\label{Pseudo-t-norm}
Let $(X,\unlhd,\wedge,\vee,0,1)$ be a bounded trellis. A binary operation $T: X^2 \to X$ is called a pseudo-triangular norm (pseudo-t-norm, for short), if it is commutative, weakly-increasing, weakly-associative and has $1$ as neutral element, i.e., $T(1, x)=x$, for any $x \in X$.
\end{definition}

Analogously, we define a pseudo-triangular conorm on a bounded trellis.

\begin{definition}\label{Pseudo-t-conorm}
Let $(X,\unlhd,\wedge,\vee,0,1)$ be a bounded trellis. A binary operation $S: X^2 \to X$ is called a pseudo-triangular conorm (pseudo-t-conorm, for short), if it is commutative, weakly-increasing, weakly-associative and has $0$ as neutral element, i.e., $S(0, x)=x$, for any $x \in X$.
\end{definition}

Next, we give some examples  of pseudo-t-norms and  pseudo-t-conorms bounded trellises. 

\begin{example}
Let $(X,\unlhd ,\wedge ,\vee )$ be a bounded trellis. It holds that 
\begin{enumerate}[label=(\roman*)]
\item $\wedge$ is a pseudo-t-norm on $X$; 
\item $\vee$ is a pseudo-t-conorm on $X$;
\item The binary operations $T_D$ (resp. $S_D$) defined in Example~\ref{T_D and S_D} is a pseudo-t-norm (resp. pseudo-t-conorm).
\end{enumerate}
\end{example}

\begin{notation}
Let $(X,\unlhd ,\wedge ,\vee )$ be a bounded trellis. We denote by:
\begin{enumerate}[label=(\roman*)]
\item $\mathcal{WAO}_1(X)$:  the class (or the set) of all pseudo-t-norms on $ X$;
\item $\mathcal{WAO}_0(X)$:  the class (or the set) of all pseudo-t-conorms on $ X$.
\end{enumerate}
\end{notation}

\begin{remark}
\begin{enumerate}[label=(\roman*)]
\item $\mathcal{WAO}_1(X)$ (resp. $\mathcal{WAO}_0(X)$) extends the class of all t-norms (resp. the class of all t-conorms) on the bounded trellis $X$.
\item In general, one can easily observe that $\mathcal{AO}_e(X)\subseteq \mathcal{WAO}_e(X)$, for any $e\in \{0,1\}$.
\end{enumerate}
\end{remark}




\subsection{Properties of pseudo-triangular norms on  bounded trellises}
In this subsection, we investigate some properties of $\mathcal{WAO}_1(X)$ and $\mathcal{WAO}_0(X)$. 

The following Proposition shows the duality between the two classes $\mathcal{WAO}_1(X)$ and $\mathcal{WAO}_0(X)$. We recall that 
for a given bounded trellis $(X,\unlhd,\wedge,\vee,0,1)$,  its dual bounded trellis  is defined as $(X^*,\unlhd^*,\wedge^*,\vee^*, 0^*,1^*)$, where $X^* = X$, $x\unlhd^* y$ if and only if $y \unlhd x$, $0^* = 1$ and $1^* = 0$.

\begin{proposition}\label{}
 Let $(X,\unlhd,\wedge,\vee,0,1)$ be a bounded trellis and $F$  a binary operation on $X$. Then the following implications hold:
 \begin{enumerate}[label=(\roman*)]
\item If $F\in \mathcal{WAO}_1(X)$, then $F\in \mathcal{WAO}_0(X^*)$; 
\item If $F\in \mathcal{WAO}_0(X)$, then $F\in \mathcal{WAO}_1(X^*)$.
\end{enumerate}
\end{proposition}

\begin{proof}
The proof is straightforward.
\end{proof}

\begin{proposition}\label{pseudo-t-norm is conjunctive}
Let $(X,\unlhd,\wedge,\vee,0,1)$ be a bounded trellis. The following implications hold: 
 \begin{enumerate}[label=(\roman*)]  
\item  Any element of $\mathcal{WAO}_1(X)$ is conjunctive;
\item  Any element of $\mathcal{WAO}_0(X)$ is disjunctive.
\end{enumerate}
\end{proposition}

\begin{proof}
\begin{enumerate}[label=(\roman*)]  
\item Let $T \in \mathcal{WAO}_1(X)$ and  $x,y\in X$. Since $1 \in X^{tr}$ and $T$ is weakly-increasing and commutative, it follows that  $T(x,y) \unlhd T(1,y)$ and $T(x,y) \unlhd T(x,1)$. The fact that $1$ is the neutral element of $T$ implies that $T(x,y) \unlhd y$ and $T(x,y) \unlhd x$. Thus, $T(x,y) \unlhd x \wedge y$. Therefore, $T$ is conjunctive.
\item The proof is dual to that of (i).
\end{enumerate}
\end{proof}
 
\begin{proposition}\label{boundary of Pseudo-t}
Let $(X,\unlhd,\wedge,\vee,0,1)$ be a bounded trellis and $F$ a binary operation on $X$. Then the following implications hold:  
 \begin{enumerate}[label=(\roman*)]  
\item If $F\in \mathcal{WAO}_1(X)$, then $F(x,0)=0$, for any $x \in X$;
\item If $F\in \mathcal{WAO}_0(X)$, then $F(x,1)=1$, for any $x \in X$. 
\end{enumerate} 
\end{proposition}

\begin{proof}{\ }
\begin{enumerate}[label=(\roman*)]
\item Suppose that $F\in \mathcal{WAO}_1(X)$ and $x\in X$. Since $1 \in X^{tr}$, $x \unlhd 1$ and $F$ is weakly-increasing, it follows that $F(x,0) \unlhd F(1,0)=0$. Thus, $F(x,0)=0$, for any $x \in X$.
\item The proof is dual to that of (i).
\end{enumerate} 
\end{proof}

\begin{remark}
If the cardinal of $X$ is greater than $1$ (i.e., $| X | > 1$), then
$$\mathcal{WAO}_0(X) \cap \mathcal{WAO}_1(X)=\emptyset.$$
\end{remark}

\subsection{Psoset structures of $\mathcal{WAO}_0(X)$ and $\mathcal{WAO}_1(X)$}
In this subsection, we discuss the bounded psoset structures of $\mathcal{WAO}_1(X)$ and $\mathcal{WAO}_0(X)$.

\bigskip
 
For any  $F_1,F_2\in \mathcal{WAO}_e(X)$, we define:
$$F_1 \unlhd_\mathcal{WAO} F_2 \text{ if and only if }  \ F_1(x,y)\unlhd F_2(x,y), \text{ for any } x, y \in X.$$

The following result is a natural generalization to that of triangular norms on the trellis.

\begin{proposition}\label{}
Let $(X,\unlhd_X,\wedge_X,\vee_X, 0, 1)$ be a bounded trellis. Then it holds that:
\begin{enumerate}[label=(\roman*)]  
\item $ T_D \unlhd_\mathcal{WAO}  F \unlhd_\mathcal{WAO}  \wedge$, for any $F \in \mathcal{WAO}_1(X)$\,;
\item $ \vee \unlhd_\mathcal{WAO}  F \unlhd_\mathcal{WAO}  S_D$, for any $F \in \mathcal{WAO}_0(X)$\,.
\end{enumerate}
\end{proposition}

\begin{proof}
\begin{enumerate}[label=(\roman*)]
\item On the one hand, Proposition~\ref{pseudo-t-norm is conjunctive} guarantees that $F \unlhd_\mathcal{WAO}  \wedge$, for any $F \in \mathcal{WAO}_1(X)$. On the other hand, $ T_D(x,y)=0 \unlhd  T(x,y)$, for any $(x,y) \in (X\setminus\{ 1 \})^2$. If $x =1$ (resp.\ $y =1$), it holds that $ T_D(1,y)=y=  F(1,y)$  (resp.\ $ T_D(x,1)=x=  F(x,1)$ ). Hence, $ T_D(x,y) \unlhd  F(x,y)$, for any $x,y \in X$. Thus, 
 $ T_D \unlhd_\mathcal{WAO}  F \unlhd_\mathcal{WAO}  \wedge$, for any $F \in \mathcal{WAO}_1(X)$\,.
\item The proof is dual to that of (i).
\end{enumerate}
\end{proof}

In a bounded Trellis $(X,\unlhd_X,\wedge_X,\vee_X, 0, 1)$, the structures $(\mathcal{WAO}_1(X),\unlhd_\mathcal{WAO} , T_D,\wedge)$ and $(\mathcal{WAO}_0(X),\unlhd_\mathcal{WAO} , \vee,S_D)$ are bounded psosets.

\bigskip

\begin{remark}
The bounded psosets  $(\mathcal{WAO}_1(X),\unlhd_\mathcal{WAO} , T_D,\wedge)$ and $(\mathcal{WAO}_0(X),\unlhd_\mathcal{WAO} , \vee,S_D)$ are not necessary  bounded trellises, since the meet (resp. the join) of any two elements is not necessary an element of $\mathcal{WAO}_1(X)$ or $\mathcal{WAO}_0(X)$. 
\end{remark}

\bigskip

The following proposition shows a case when an element of $\mathcal{WAO}_1(X)$ (resp. an element of $\mathcal{WAO}_0(X)$) coincides with the meet (resp. the join) operation. It is particular case of the weaker types of increasing binary operations on a bounded trellis that coincide with the meet (resp. the join) operation.

 \bigskip
 
\begin{proposition}\label{characterizations of the meet and the join}
Let $(X,\unlhd,\wedge,\vee)$ be a bounded trellis and $F$ a binary operation on $X$. The following statements hold:
 \begin{enumerate}[label=(\roman*)]
\item If $F\in \mathcal{WAO}_1(X)$, idempotent and satisfying $F(x \wedge y,x \wedge y) \unlhd F(x,y)$, for any $x,y \in X$, then $F$ is the meet $(\wedge)$ operation of $X$;
\item If $F\in \mathcal{WAO}_0(X)$, idempotent and satisfying $F(x,y) \unlhd F(x \vee y,x \vee y)$, for any $x,y \in X$, then $F$ is the join $(\vee)$ operation  of $X$.
\end{enumerate}  
\end{proposition}

\begin{proof}
 \begin{enumerate}[label=(\roman*)]
\item  On the one hand, since $F\in \mathcal{WAO}_1(X)$ which means that $F$ is conjunctive, it holds that  $F(x,y) \unlhd x \wedge y$, for any $x, y \in X$.  On the other hand, the fact that $F$ is idempotent and satisfying $F(x \wedge y,x \wedge y) \unlhd F(x,y)$, for any $x,y \in X$ implies that $ x \wedge y = F(x \wedge y,x \wedge y ) \unlhd F(x,y) $. Thus, $F$ is the meet operation $(\wedge) $ of $X$. 
\item The proof is dual to that of (i).
\end{enumerate} 
\end{proof}

\begin{remark}
 The converse of the above Proposition \ref{characterizations of the meet and the join} is immediate.
\end{remark}

\section{Constructions of some elements of $\mathcal{WAO}_0(X)$ and $\mathcal{WAO}_1(X)$}
In this section, we construct some elements of $\mathcal{WAO}_1(X)$ and $\mathcal{WAO}_0(X)$ on  bounded trellises.

\bigskip

Let $(X,\unlhd,\wedge,\vee,0,1)$ be a bounded trellis and $e \in X$. Let $T_e$ and $S_e$  two binary operations  on $X$ defined as follows:

\begin{center}
$T_e(x,y)=\left\{\begin{array}{ll}{x \wedge y} & \text{if } x=1 \text{ or }  y=1;\\
 {(x \wedge y) \wedge e} & \text{otherwise;}\end{array}\right.$
\end{center}
 and
\begin{center}
$S_e(x,y)=\left\{\begin{array}{ll}{x \vee y} & \text{if }  x=0  \text{ or } y=0; \\ 
{(x \vee y) \vee e} & \text{otherwise.} \end{array}\right.$
\end{center}

\bigskip

\begin{remark}\label{the pseudo-t-norm f_e}
In general, $T_e$ (resp.\ $S_e$) is not necessarily an element of $\mathcal{WAO}_1(X)$ (resp.\ $\mathcal{WAO}_0(X)$). Indeed, let $(X=\{ 0,a,b,c,d,e,f,1\},\unlhd, \wedge,\vee,0,1)$ be a bounded trellis given by the Hasse  diagram in Figure~\ref{fig3}.
\begin{figure}[H]
\[\begin{tikzpicture}
\tikzstyle{estun}=[>=latex,thick,dotted]
    \vertex[fill] (0) at (0,0)    [label=right:$ 0$]  {};
    \vertex[fill] (a) at (0,0.75)        [label=right:$a $]  {};
    \vertex[fill] (b) at (0.75,1.5)  [label=right:$ b$]  {};
    \vertex[fill] (c) at (0.75,2.25)   [label=right:$c $]  {};
    \vertex[fill] (d) at (0.75,3)   [label=right:$d $]  {};
    \vertex[fill] (e) at (-0.75,3)  [label=left:$e$]  {};
    \vertex[fill] (f) at (1.5,2.25)   [label=right:$f$]  {};
    \vertex[fill] (1) at (0,3.75)     [label=left:$1 $]  {};
    
    \path
        (0) edge (a)
        (a) edge (b)
        (b) edge (c)
        (c) edge (d)
        (d) edge (1)
        (a) edge (e)
        (b) edge (f)   
        (f) edge (d)                       
        (e) edge (1) 
        ;
          \draw[estun] (a) to (c) 
   ;  
\end{tikzpicture}\]
\caption{Hasse  diagram of the trellis $(X=\lbrace 0,a,b,c,d,e,f,1 \rbrace,\unlhd)$.}\label{fig3}
\end{figure}
  Sitting $x=f$ and $y=d$, then $x \unlhd y$ and $(x ,y) \in  (X^{tr})^2$. Since $T_e(f,c)= (f \wedge c) \wedge e= a \ntrianglelefteq T_e(d,c)=(d \wedge c) \wedge e =0$,  it follows that  $T_e$ is not  weakly-increasing. Therefore, $T_e \notin \mathcal{WAO}_1(X)$.
\end{remark} 

In view of remark~\ref{the pseudo-t-norm f_e}, we give sufficient conditions under which the binary operation  $T_e$ is an element of $\mathcal{WAO}_1(X)$. 

\begin{proposition}\label{T_e is a pseudo-t-norm}
Let $(X,\unlhd,\wedge,\vee ,0,1)$ be a bounded trellis. The following implications hold:
\begin{enumerate}[label=(\roman*)] 
\item If $e \in X^{\wedge-ass}$, then $T_e \in \mathcal{WAO}_1(X)$;
\item If $e \in X^{\vee-ass}$, then $S_e \in \mathcal{WAO}_0(X)$.
\end{enumerate}
\end{proposition}

\[
\begin{tikzcd}[column sep = 2em, row sep = 2em]
& 1 \arrow[dash]{dl} \arrow[dash]{dr} & \\
a \arrow[dash]{dr} & & b \arrow[dash]{dl} \\
& 0 &
\end{tikzcd}
\]

\begin{proof}{\ }
We only give the proof of (i), as the proof of (ii) is similar. One  easily verifies that $T_e$ is commutative and satisfies the boundary condition. Now, let $(x,y) \in X \times X^{tr}$ such that $ x \unlhd y$ and $z \in X$. Then we discuss the following two possible cases: 
 \begin{enumerate}[label=(\roman*)]  
\item If $z=1$, then $T_e(x,z)=x \unlhd y= T_e(y,z)$.
\item If $z \neq 1$, then  we have three possible cases: 
 \begin{enumerate}[label=(\roman*)]  
\item If $x = 1$, then $y=1$ and $T_e(x,z)=z \unlhd z= T_e(y,z)$.
\item If $x \neq 1$ and $y=1$, then the fact that $e \in X^{\wedge-ass}$ implies that  $T_e(x,z)=(x \wedge z) \wedge e= (x \wedge e) \wedge z \unlhd z = T_e(y,z)$. Thus, $T_e(x,z)  \unlhd  T_e(y,z)$.
\item If $y \neq 1$, then $T_e(x,z)=(x \wedge z) \wedge e$ and $T_e(y,z)=(y \wedge z) \wedge e$. Since $ x \unlhd y$ and $y \in X^{tr}$, it follows that $x \wedge z \unlhd y \wedge z$. The fact that  $e \in X^{\wedge-ass}$ implies  $(x \wedge z) \wedge e \unlhd (y \wedge z) \wedge e$. Thus, $T_e(x,z)  \unlhd  T_e(y,z)$. 
\end{enumerate}
\end{enumerate}
Therefore, $T_e$ is  weakly-increasing.

Now, we prove that $T_e$ is weakly-associative.  Let $ x ,y,z \in X$ such that $[x,y,z] \in X^{\wedge-ass}$.  Since $e \in X^{\wedge-ass}$, it holds that
\begin{align*}
T_e(x,T_e(y,z))&= (x \wedge((y \wedge z) \wedge e)) \wedge e\\
&= ((x \wedge(y \wedge z)) \wedge e) \wedge e\\
&= (((x \wedge y) \wedge z) \wedge e) \wedge e\\
&=(((x \wedge y) \wedge e) \wedge z) \wedge e\\
&= (T_e(x,y)\wedge z) \wedge e\\
&=T_e(T_e(x,y),z)\,.
\end{align*}
  Hence, $T_e$ is weakly-associative. Therefore, $T_e \in \mathcal{WAO}_1(X)$.
\end{proof}

\begin{remark}
Particular cases: since $0,1 \in X^{\wedge-ass}$, we recognize that 
\begin{enumerate}[label=(\roman*)] 
\item $T_0 = T_D$ and $T_1 = \wedge$;
\item $S_0 = \vee$ and $S_1 = S_D$.
\end{enumerate}
\end{remark}

\begin{proposition}{\ }\label{}
Let $\mathbb{X}=(X,\unlhd,\wedge,\vee,0,1)$ be a bounded modular trellis and the binary operations  $Z$ and $Z^*$ defined as follows:
 
$Z(x, y)=\left\{\begin{array}{ll}{x \wedge y} & {\text { if } x \vee y=1};\\
 {0} & {\text { otherwise; }}\end{array}\right.$
 and
$Z^*(x, y)=\left\{\begin{array}{ll}{x \vee y} & {\text { if } x \wedge y=0}; \\ 
{1} & {\text { otherwise;}}\end{array}\right.$

Then, $Z \in \mathcal{WAO}_1(X)$ and $Z^* \in \mathcal{WAO}_0(X)$.

\end{proposition}

\begin{proof}{\ }
 The proof is similar to that of $Z^*$. One easily verifies that $T_\mathrm{Z}$ is commutative and satisfies the boundary conditions. Now, let $(x,y) \in X \times X^{tr}$ such that $ x \unlhd y$. Then we discuss the following two possible cases: 
 \begin{enumerate}[label=(\roman*)]  
\item If $T_\mathrm{Z}(x,z)=0$, then $T_\mathrm{Z}(x,z)=0 \unlhd T_\mathrm{Z}(y,z)$, for any $z \in X$.
\item If $T_\mathrm{Z}(x,z)=x \wedge z$, then $x \vee z =1$. Proposition~\ref{the meet and the join are p-increasing} guarantees that $x \vee z \unlhd y \vee z$, for any $z \in X$. Thus, $y \vee z=1$ and $T_\mathrm{Z}(y,z)= y \wedge z$. Since $y \in X^{tr}$, it holds that $T_\mathrm{Z}(x,z)=x \wedge z \unlhd y \wedge z= T_\mathrm{Z}(y,z)$, for any $z \in X$.  
\end{enumerate}
Hence, $T_\mathrm{Z}$ is  weakly-increasing. Now, we prove that $T_\mathrm{Z}$ is weakly-associative.  Let $ x ,y,z \in X$ such that $[x,y,z] \in X^\wedge$. On the one hand, we have that\\
$T_\mathrm{Z}\left(x, T_\mathrm{Z}(y, z)\right)$
$\quad=\left\{\begin{array}{ll}{x \wedge y \wedge z} & {\text { if } y \vee z=1 \text { and } x \vee(y \wedge z)=1}, \\ 
{0} & {\text { otherwise. }}\end{array}\right.$\\
On the other hand, it holds that\\
$T_\mathrm{Z}\left(T_\mathrm{Z}(x, y), z\right)$
$\quad=\left\{\begin{array}{ll}{x \wedge y \wedge z} & {\text { if } x \vee y=1 \text { and } z \vee (x \wedge y)=1}, \\ {0} & {\text { otherwise. }}\end{array}\right.$\\
We will show that $y\vee z=1$  and $x \vee (y \wedge z)=1$ implies $x \vee y=1$ and $z \vee (x \wedge y)=1$. The proof of the converse implication is similar. Firstly, since $y \unlhd x \vee y$ and $y \vee z=1$. Proposition~\ref{two famous implications on modular trellis} guarantees that $ y \wedge z \unlhd x \vee y $. Thus, $1= x \vee (y \wedge z) \unlhd x \vee y$. Hence, $x \vee y=1$. Secondly, since $x \vee(y \wedge z)=1$, it holds that $y= y \wedge(x \vee (y \wedge z))$. The fact that $\mathbb{X}$ is  modular  implies that
$y=y \wedge(x \vee (y \wedge z))=(x \wedge y) \vee(y \wedge z)= (x \wedge y) \vee z )\wedge y$. Thus,  
$y \unlhd(x \wedge y) \vee z$. On the other hand, since 
$z \unlhd(x \wedge y) \vee z$, it follows that  $y \vee z \unlhd (x \wedge y) \vee z$. Hence, $(x \wedge y) \vee z=1$. Since $[x,y,z] \in X^\wedge$, it holds that $x \meet (y \meet z) = (x \meet y) \meet z$.  Hence, $T_\mathrm{Z}$ is weakly-associative. Therefore, $T_\mathrm{Z}$ is a pseudo-t-norm on $X$ (i.e., $Z \in \mathcal{WAO}_1(X)$).
\end{proof}

In the following result, we propose a new ordinal sum construction of $\mathcal{WAO}_1(X)$ and $\mathcal{WAO}_0(X)$ on  bounded trellises according to~\cite{ertuugrul2015modified}. We start by the following immediate proposition.
 
\begin{proposition}
Let $(X,\unlhd,\wedge,\vee)$ be a trellis and $a, b \in X^{ass}$ such that  $a \unlhd b$. The following subintervals of $X$ defined as:
$$ [a, b]=\{x \in X | a \unlhd x \unlhd b\},$$
$$(a, b]=\{x \in X | a \lhd x \unlhd b\},$$
$$[a, b)=\{x \in X | a \unlhd x \lhd b\},$$ 
 $$(a, b)=\{x \in X | a \lhd x \lhd b\} ,$$
 are subtrellises of $X$.
\end{proposition} 

\begin{theorem}\label{construction 01}
Let  $(X,\unlhd,\wedge,\vee,0,1)$ be a bounded trellis and $a \in X^{ass} \backslash\{0,1\}$. If  $V:[a, 1]^{2} \rightarrow[a, 1]$ an element of $\mathcal{WAO}_1([a, 1])$  and $W:[0, a]^{2} \rightarrow[0, a]$ an element of  $\mathcal{WAO}_0([0, a])$, then the binary operations $T: X^{2} \rightarrow X$ and $S: X^{2} \rightarrow X$ defined as follows:
\begin{center}
 $T(x, y)=\left\{\begin{array}{ll}{x \wedge y} & {\text {if } x=1 \text { or } y=1}; \\
 {V(x, y)} & {\text {if } x, y \in[a, 1)}; \\
  {x \wedge y \wedge a} & {\text {otherwise;}}\end{array}\right.$
 \end{center} 
 and 
\begin{center}
$S(x, y)=\left\{\begin{array}{ll}{x \vee y } & {\text {if } x=0 \text { or } y=0}; \\ 
{W(x, y) } & {\text {if } x, y \in(0, a]}; \\
 {x \vee y \vee a } & {\text {otherwise;}}\end{array}\right.$
\end{center}
are elements of $\mathcal{WAO}_1(X)$ and $\mathcal{WAO}_0(X)$, respectively.
\end{theorem}

\begin{proof}{\ }
The proof is similar to that of $S$. One  easily verifies that $T$ is commutative and satisfies the boundary conditions. Now, let $x,y \in X \times X^{tr}$ such that $x \unlhd y$. Then we discuss the following possible cases: 
 \begin{enumerate}[label=(\roman*)] 
 \item If $x=1$ or $z=1$, then $T(x, z)=T(y, z)$.
\item If $x,z \in [a, 1)$, then, also $a \unlhd y$ and $T(x, z)=V(x, z) \unlhd V(y, z)=T(y, z)$.
\item If $x \notin [a, 1)$ and $z \in [a, 1)$, it holds that $T(x, z)=x \wedge a$ and we have three possible cases:
 \begin{enumerate}[label=(\roman*)] 
 \item If $y=1$, then $T(y, z)=z $. Since $a \in X^{ass}$, then $T(x, z)=x \wedge a \unlhd z=T(y, z)$.
\item If $y \in [a, 1)$, then $T(y, z)=V(y,z) \in [a, 1)$. Since, $a \in X^{ass}$, it follows that  $T(x, z)=x \wedge a \unlhd V(y,z)=T(y, z)$.
\item If $y \notin [a, 1]$, then $T(y, z)=y \wedge a$. Since $a \in X^{ass}$, then it follows that $T(x, z)=x \wedge a  \unlhd y \wedge a= T(y, z)$.
\end{enumerate}
\item If $x \notin [a, 1]$ and $z \notin [a, 1]$, then $T(x, z)=x \wedge z \wedge a$ and $T(y, z)= y \wedge z \wedge a$. Thus, Proposition~\ref{the pseudo-t-norm f_e} guarantees that $T(x, z)=x \wedge z \wedge a \unlhd y \wedge z \wedge a = T(y, z)$.
 \end{enumerate}
 Hence, $T$ is weakly-increasing. Next, we prove that $T$ is weakly-associative. Let $ x ,y,z \in X$ such that $[x,y,z] \in X^\wedge$. The proof is split into all possible cases.
\begin{enumerate}[label=(\roman*)] 
 \item If $x,y \in [a, 1)$, then we have two cases:
  \begin{enumerate}[label=(\roman*)] 
  \item[(a)] If $z \in [a, 1)$, then:
\begin{center}
 $\begin{aligned} 
 T(x, T(y, z)) &=T(x, V(y, z)) \\
  &=V(x, V(y, z)) \\ &=V(V(x, y), z) \\ 
  &=T(V(x, y), z) \\ 
  &=T(T(x, y), z).
   \end{aligned}$
\end{center}
 \item[(b)] If $z \in X  \setminus [a, 1)$, then:
\begin{center} 
 $\begin{aligned}
  T(x, T(y, z)) &=T(x, y \wedge z \wedge a) \\
  &=x \wedge (y \wedge z \wedge a) \wedge a  \\
   &=z \wedge a  \text{\quad (car, } a \in X^{ass})\\
    &=V(x, y) \wedge z \wedge a \\ 
    &=T(V(x, y), z) \\
     &=T(T(x, y), z).
      \end{aligned}$
\end{center}
 \end{enumerate}
 \item If $x \in [a, 1)$ and $y \in X \setminus [a, 1)$, then we have two cases:
 \begin{enumerate}[label=(\roman*)]
  \item[(a)] If $z \in [a, 1)$, then this case have been studied in $(i.b)$.
 \item[(b)] If $z \in X \setminus [a, 1)$, then:
\begin{center}
 $\begin{aligned} 
 T(x, T(y, z)) &=T(x, y \wedge z \wedge a) \\ 
 &=x \wedge (y \wedge z \wedge a) \wedge a\\
  &=(x \wedge y \wedge a) \wedge z \wedge a \text{\quad (car, } a \in X^{ass})\\
  &=T(x \wedge y \wedge a, z)\\
   &=T(T(x, y), z). 
   \end{aligned}$
\end{center}
 \end{enumerate}

 \item If $x,y \in X \setminus [a, 1)$, then we have two cases: 
 \begin{enumerate}[label=(\roman*)]
\item[(a)] If $z \in [a, 1)$, then this case have been studied in $(ii.b)$.
\item[(b)] If $z \in X \setminus [a, 1)$, then: 
\begin{center}
$\begin{aligned}
 T(x, T(y, z)) &=T(x, y \wedge z \wedge a) \\
  &=x \wedge y \wedge z \wedge a \\ 
  &=T(x \wedge y \wedge a, z) \\
   &=T(T(x, y), z).
   \end{aligned}$
\end{center}
\end{enumerate}
 \end{enumerate}
Hence, $T$ is weakly-associative on $X$. Therefore, $T \in \mathcal{WAO}_1(X)$.
\end{proof}

One easily Observes that  $T$ and  $S$ on a bounded trellis considered in Theorem~\ref{construction 01} can be described as  follows:
\begin{center}
$T(x, y)=\left\{\begin{array}{ll}{
V(x, y)} & {\text { if }(x, y) \in[a, 1)^{2}}; \\
 {y \wedge a} & {\text { if } x \in[a, 1), y \| a}; \\
  {x \wedge a} & {\text { if } y \in[a, 1), x \| a};\\
   {x \wedge y \wedge a} & {\text { if } x\|a, y\| a}; \\
    {x \wedge y} & {\text { otherwise; }}
    \end{array}\right.$\end{center}
    and
\begin{center}
    $S(x, y)=\left\{\begin{array}{ll}{
    W(x, y)} & {\text { if }(x, y) \in(0, a]^{2}}; \\ 
    {y \vee a} & {\text { if } x \in(0, a], y \| a}; \\
     {x \vee a} & {\text { if } y \in(0, a], x \| a}; \\
      {x \vee y \vee a} & {\text { if } x\|a, y\| a}; \\
       {x \vee y} & {\text { otherwise.}}
       \end{array}\right.$
\end{center}

Thus, we get $T$ and  $S$ by the next figures:
\begin{center}
\includegraphics[width=6cm]{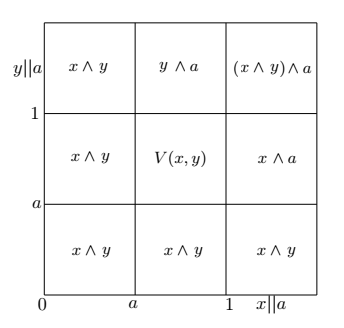}  
\end{center} 

\section{T-distributivity on  bounded trellises}
In this section, we introduce the notion of $F$-distributivity for an arbitrary binary operation $F$ on a bounded trellis and we determine a relationship between  t-norms and pseudo-t-norms. Moreover, we characterize pseudo-t-norms (resp.\ pseudo-t-conorms) on a bounded trellis with respect to the $F$-distributivity.

\begin{definition} 
 Let $(X,\unlhd,\wedge,\vee,0,1)$ be a bounded trellis and $F_{1}$ and $F_{2}$ two binary operations on $X$. If $F_{1}\left(x, F_{2}(y, z)\right)=F_{2}\left(F_{1}(x, y), F_{1}(x, z)\right)$, for any $x, y, z \in X$ where at least one of the elements $y, z$ is not $1$ or not $0$ (i.e., $(0,0) \neq (y,z)$ and $(1,1) \neq (y,z))$, then $F_{1}$ is distributive over $F_{2}$ ($F_{1}$ is $F_{2}$-distributive, for short).

\end{definition}

\begin{proposition}\label{}
Let $(X,\unlhd,\wedge,\vee,0,1)$ be a bounded trellis and $T^*$ is a pseudo-t-norm on $X$. If $T^*$ is $T$-distributive, for any $T \in \mathcal{WAO}_1(X)$, then $T^* \in \mathcal{AO}_1(X)$.
\end{proposition}

\begin{proof}{\ }
Let $x,y \in X$ such that $x \unlhd y$. The fact that the meet operation $(\wedge)$ is a pseudo-t-norm on $X$ implies that  $T^*(x,z) \wedge T^*(y,z)= T^*(x \wedge y,z)= T^*(x,z)$. Thus, $T^*$ is increasing. Now, we prove that $T^*$ is associative. Since $T^*$ is a pseudo-t-norm on $X$ and increasing, it holds that  $T^*(x,T^*(y,z))= T^*(T^*(x,y),T^*(x,z)) \unlhd T^*(T^*(x,y), z)$ and $T^*(T^*(x,y), z)=T^*(T^*(x,z),T^*(y,z))\unlhd T^*(x,T^*(y,z))$, for any $x,y,z \in X$. Thus, $T^*(x,T^*(y,z))=T^*(T^*(x,y), z)$, for any $x,y,z \in X$. Hence, $T^*$ is associative. Since $T^*$ is commutative and has $1$ as a neutral element, it follows that $T^*$ is a t-norm on $X$.
\end{proof}

In the following proposition, we give a  $T$-distributive t-norm, for any pseudo-t-norm $T$ on bounded trellis.
\begin{proposition}\label{} 
Let  $(X,\unlhd,\wedge,\vee)$ be a bounded trellis. Then the smallest pseudo-t-norm (t-norm) $T_D$  is $T$-distributive, for any $T \in \mathcal{WAO}_1(X)$.
\end{proposition}

\begin{proof}{\ }
 Let $T$ be an arbitrary pseudo-t-norm on $X$. We must show that the equality
        $$T_D(x, T(y, z))=T\left(T_D(x, y), T_D(x, z)\right) $$

holds for every element $x, y, z$ of $X$ such that  $y \neq 1$ or $z \neq 1$. Suppose that $z \neq 1$. Then we discuss the following two possible cases: 
 \begin{enumerate}[label=(\roman*)]  
\item If $x=1$, then  $T_D(x, T(y, z))=T(y, z)=T(T_D(x, y), T_D(x, z))$.
\item If $x\neq 1$, we have two possible cases:
    \begin{enumerate}[label=(\roman*)]  
    \item If $y=1$, then $T_D(x, T(y, z))=T_D(x, z)=0$ and $T(T_D(x, y), T_D(x, z))=T(x,0)$. Proposition~\ref{boundary of Pseudo-t} guarantees that  $T(x,0)=0$, for any $x \in X$. Then the equality is holds.
	\item If $y \neq 1$, then $T(y, z) \unlhd y \unlhd 1$ and $y \neq 1$, for any pseudo-t-norm $T$. Thus, $T(y, z) \neq 1$. Hence, $T_D(x, T(y, z))=0$ and $T\left(T_D(x, y), T_D(x, z)\right)=T(0,0)=0$.
    \end{enumerate}
\end{enumerate}
Therefore, $T_D$ is $T$-distributive, for any $T \in \mathcal{WAO}_1(X)$.
\end{proof}

\begin{proposition}\label{pseudo-t-norm weaker than t-norm}
Let $(X,\unlhd,\wedge,\vee,0,1)$ be a bounded trellis, $T$ a pseudo-t-norm and $T^*$ is a t-norm on $X$. If $T$ is $T^*$-distributive, then $T  \unlhd_\mathcal{WAO} T^*$ (i.e., $T$ is weaker that $T^*$).
\end{proposition}

\begin{proof}{\ }
Let $x,y \in X$, then $T(x,y)=T(T^*(x,1),y)= T^*(T(x,y),T(1,y))=T^*(T(x,y),y)$. Since $T^*$ is a t-norm, it holds that $T^*(T(x,y),y) \unlhd T^*(x,y)$. Thus $T(x,y) \unlhd_\mathcal{WAO} T^*(x,y)$, for any $x,y \in X$. Hence, $T$ is weaker that $T^*$.
\end{proof}

\begin{proposition}\label{}
Let $(X,\unlhd,\wedge,\vee,0,1)$ be a bounded trellis and $T^*$ is a t-norm on $X$. If $T$ is $T^*$-distributive, for any $T \in \mathcal{WAO}_1(X)$, then $T^*= \wedge $. Moreover,  $(X,\unlhd,\wedge,\vee)$ is a lattice.
\end{proposition}

\begin{proof}{\ }
Let an arbitrary pseudo-t-norm $T$ and $T^*$ is a t-norm on $X$, then Proposition~\ref{pseudo-t-norm weaker than t-norm} guarantees that  $T \unlhd_\mathcal{WAO} T^*$. Suppose that $T=\wedge$, then $\wedge \unlhd_\mathcal{WAO} T^*$. Since $T^*$ is conjunctive, it holds that $\wedge = T^*$. Thus, $\wedge$ is associative. Hence, $(X,\unlhd,\wedge,\vee)$ is a lattice.
\end{proof}

\begin{proposition}\label{characterise of the meet}
Let $(X,\unlhd,\wedge,\vee,0,1)$ be a bounded trellis,  $T$ is a pseudo-t-norm and $S$ is a pseudo-t-conorm on $X$. The following implications hold:
\begin{enumerate}[label=(\roman*)]
\item If $S$ is $T$-distributive, then $T$ is idempotent; 
\item If $T$ is $S$-distributive, then $S$ is idempotent.
\end{enumerate}
\end{proposition}

\begin{proof}{\ }
\begin{enumerate}[label=(\roman*)]
\item Since $T$ is pseudo-t-norm and $S$ is pseudo-t-conorm on $X$, it follows that $x= S(x,0)=S(x, T(0,0))=T(S(x,0),S(x,0))=T(x,x)$, for any $x \in X$. Thus, $T$ is idempotent.
\item The proof is similar to that of (i). 
\end{enumerate}
\end{proof}

Proposition~\ref{characterise of the meet} leads to the following corollary.

\begin{corollary}\label{}
Let $(X,\unlhd,\wedge,\vee,0,1)$ be a bounded trellis,  $T$ is a pseudo-t-norm and $S$ is a pseudo-t-conorm on $X$. Then it holds that:
\begin{enumerate}[label=(\roman*)]
\item If $S$ is $T$-distributive and $T$ satisfying $T(x \wedge y,x \wedge y) \unlhd T(x,y)$, for any $x,y \in X$, then $T= \wedge$. 
\item If $T$ is $S$-distributive and $S$ satisfying $S(x \vee y,x \vee y) \unlhd S(x,y)$, for any $x,y \in X$, then $S= \vee$. 
\end{enumerate} 
\end{corollary}

\section{Relationship among $\mathcal{WAO}_e(X)$ and  isomorphisms on  bounded trellises}
In this section, we conjugate  elements of $\mathcal{WAO}_1(X)$  (resp. elements of $\mathcal{WAO}_0(X)$) and an isomorphism map on a bounded trellis $X$.  First, we start by the following proposition.

\begin{proposition}\label{rho is intern of X^{lt}}
Let $(X_1, \unlhd_{1}, \wedge_{1}, \vee_{1})$, $(X_2, \unlhd_{2}, \wedge_{2}, \vee_{2})$ be two trellises and $\rho:X_1 \longrightarrow X_2$ an  isomorphism. Then $\rho (X_1^{tr}) \subseteq X_2^{tr}$\,.
\end{proposition} 

\begin{proof} 
Let $x , y \in X_2$ and  $z \in \rho (X_1^{tr})$ such that $x \unlhd_{2} y \unlhd_{2} z$. Then there exist $x', y' \in X_1$ and $z' \in X_1^{tr}$ such that $\rho(x') \unlhd_{2} \rho(y') \unlhd_{2} \rho(z')$. From the increasingness of $\rho^{-1}$,  it holds  that $x' \unlhd_{1} y' \unlhd_{1} z'$.  The fact that $z' \in X_1^{tr}$ implies that $x' \unlhd_{1} z'$, i.e., $x' \wedge_{1} z' =x'$. Since $\rho$ is homomorphism; it follows that $\rho(x') \wedge_{2} \rho(z') = \rho(x' \wedge_{1} z')= \rho(x')$. Hence, $\rho(x') \unlhd_{2}  \rho(z')$, i.e., $x \unlhd_{2} z$. Thus,  $ z \in X^{tr}$. Therefore, $\rho (X_1^{tr}) \subseteq X_2^{tr}$\,. 
\end{proof}

%

\begin{proposition}\label{associative element of rho}
Let $(X_1, \unlhd_{1}, \wedge_{1}, \vee_{1})$, $(X_2, \unlhd_{2}, \wedge_{2}, \vee_{2})$ be two trellises and $\rho:X_1 \longrightarrow X_2$ an  isomorphism. Then $[x,y,z]\in X_1^{\wedge_{1}}$ (resp.\ $[x,y,z]\in X_1^{\vee_{1}}$) if and only if $[\rho(x),\rho(y),\rho(z)]\in X_2^{\wedge_{2}}$ (resp.\ $[\rho(x),\rho(y),\rho(z)]\in X_2^{\vee_{2}}$).
\end{proposition}

\begin{proof} 
Let $x,y,z \in X_1$ such that  $[x,y,z]\in X_1^{\wedge_{1}}$. Since $\rho$ is an  isomorphism, then
\begin{align*}
\rho(x) \wedge_2 (\rho(y) \wedge_2 \rho(z))&=\rho(x) \wedge_2 \rho(y \wedge_1 z))\\
&=\rho(x \wedge_1 (y \wedge_1 z))\\
&=\rho((x \wedge_1 y) \wedge_1 z)\\
&= \rho(x\wedge_1 y) \wedge_2 \rho(z))\\
&=(\rho(x) \wedge_2 \rho(y)) \wedge_2 \rho(z)\,.
\end{align*}
 Therefore, $[\rho(x),\rho(y),\rho(z)]\in X_2^{\wedge_{2}}$. In a similar way, we prove that $[x,y,z]\in X_1^{\vee_{1}}$ if and only if $[\rho(x),\rho(y),\rho(z)]\in X_2^{\vee_{2}}$\,.
\end{proof}
\begin{proposition}\label{transformation of pseudo-t-norm}
$(X_1, \unlhd_{1}, \wedge_{1}, \vee_{1},0_1,1_1)$, $(X_2, \unlhd_{2}, \wedge_{2}, \vee_{2},0_2,1_2)$ be two bounded trellises, $T \in \mathcal{WAO}_1(X_2)$ and $\rho:X_1 \longrightarrow X_2$ an  isomorphism. Then the binary operation $T^\rho$ defined by:
\begin{center}
$T^\rho(x,y)= \rho^{-1}(T(\rho(x), \rho (y)))$, for any $x,y\in X_1$\,,  
\end{center}
is an element of $\mathcal{WAO}_1(X_1)$. 
\end{proposition}
\begin{proof}{\ } 
One  easily verifies that $T^\rho$ is commutative and satisfies the boundary condition. Now, let $(x,y) \in X_1 \times X_1^{tr}$ such that $x \unlhd_1 y$. Proposition~\ref{rho is intern of X^{lt}} assures that $\rho(y) \in X_2^{tr}$. Since $T$ is weakly-increasing, it holds that  $T(\rho(x), \rho (z))\unlhd_2 T(\rho(y), \rho (z))$, for any $z\in X_1$. The fact that $\rho^{-1}$  is increasing on $X_2$ implies that  $\rho^{-1}(T(\rho(x), \rho (z))) \unlhd_1 \rho^{-1}(T(\rho(y), \rho (z)))$, for any $z \in X_1$, i.e., $T^\rho(x,z) \unlhd_1 T^\rho(y,z)$, for any $z \in X_1$.  Hence, $T^\rho$ is weakly-increasing on $X_1$. Next, we prove that $T^\rho$ is weakly-associative. Let $ x ,y,z \in X_1$ such that $[x,y,z] \in X_1^{\wedge_{1}}$. Proposition~\ref{associative element of rho} assures that $(\rho(x),\rho(y),\rho(z))\in X_2^{\wedge_{2}}$. Thus
\begin{align*}
T^\rho(T^\rho(x,y),z)&=\rho^{-1}(T(\rho(T^\rho(x,y)), \rho (z)))\\
& =\rho^{-1}(T(\rho(\rho^{-1}(T(\rho(x), \rho (y)))), \rho (z)))\\
&= \rho^{-1}(T(T(\rho(x), \rho (y)), \rho (z)))\\
&= \rho^{-1}(T(\rho(x), T(\rho (y), \rho (z))))\\
& =\rho^{-1}(T(\rho(x), \rho(\rho^{-1}(T(\rho (y), \rho (z))))))\\
&=\rho^{-1}(T(\rho(x), \rho(T^\rho(y,z))))\\
&=T^\rho(x,T^\rho(y,z))\,.
\end{align*}
Hence, $T^\rho$ is weakly-associative on $X_1$. Therefore, $T^\rho \in \mathcal{WAO}_1(X_1)$\,. 
\end{proof} 

Notice that in a bounded trellis $(X,\unlhd,\wedge,\vee,0,1)$, the identity map $Id_{X}$ of $X$ (i.e., $Id_{X}(x) = x$, for any $x \in X$) is an  isomorphism (automorphism). Then   $T^{Id_{X}}=T$, for any  $T \in \mathcal{WAO}_1(X)$.

Dually, we have the following result for the elements of $\mathcal{WAO}_0(X)$.

\begin{proposition}\label{transformation of pseudo-t-conorm}
$(X_1, \unlhd_{1}, \wedge_{1}, \vee_{1},0_1,1_1)$, $(X_2, \unlhd_{2}, \wedge_{2}, \vee_{2},0_2,1_2)$ be two bounded trellises, $S \in \mathcal{WAO}_0(X_2)$ and $\rho:X_1 \longrightarrow X_2$ an  isomorphism. Then the binary operation $S^\rho$ defined by:
\begin{center}
$S^\rho(x,y)= \rho^{-1}(S(\rho(x), \rho (y)))$, for any $x,y\in X_1$\,,  
\end{center}
is an element of $\mathcal{WAO}_0(X_1)$. 
\end{proposition}

\section{Conclusion}\label{Conclusion}
In this paper, we have studied the notion of pseudo-triangular norms on a bounded trellis and provided some examples. Further, we have provided some new class of pseudo-triangular norms and some characterisation. We intend that this study open the door of different applications of trellis structure in various areas using pseudo-triangular norms.


\section*{References}\label{References}
\begin{thebibliography}{10}

\bibitem{bede2013fuzzy}
B. Bede, Fuzzy clustering, In Mathematics of Fuzzy Sets and Fuzzy Logic, Springer, 2013,  pp.~213--219.

\bibitem{bedregal2006t}
B.~C. Bedregal, H.~S. Santos and R. Callejas-Bedregal, T-norms on bounded lattices: t-norm morphisms and operators, IEEE International Conference on Fuzzy Systems,  2006, pp.~22--28.

\bibitem{de1999triangular}
B. De~Baets and  R. Mesiar,   Triangular norms on product lattices, Fuzzy Sets and Systems 104(1) (1999) 61--75.

\bibitem{de1994order}
 G. De~Cooman and E.~E. Kerre,  Order norms on bounded partially ordered sets, The Journal of Fuzzy Mathematics 2 (1994) 281--310.

\bibitem{dummit2004abstract}
D.~S. Dummit and R.~M. Foote, Abstract algebra,  Wiley Hoboken, 2004.

\bibitem{ertuugrul2015modified}
{\"U}. Ertu{\u{g}}rul,  F. Kara{\c{c}}al and R. Mesiar, 
Modified ordinal sums of triangular norms and triangular conorms on
  bounded lattices,
 International Journal of Intelligent Systems 30(7) (2015)  807--817.

\bibitem{fried1970tournaments}
 E. Fried, 
 Tournaments and non-associative lattices,
 Ann. Univ. Sci. Budapest, Sect. Math 13 (1970) 151--164.

\bibitem{fried1973some}
E. Fried and G. Gr{\"a}tzer, 
 Some examples of weakly associative lattices,
 In  Colloquium Mathematicae 27 (1973) 215--221.

\bibitem{gladstien1973characterization}
 K. Gladstien,
 A characterization of complete trellises of finite length, 
 algebra universalis 3(1) (1973) 341--344.

\bibitem{halavs2016clone}
 R. Hala{\v{s}} and J. P{\'o}cs, 
On the clone of aggregation functions on bounded lattices,
Information Sciences 329 (2016) 381--389.

\bibitem{karaccal2006direct}
 F. Kara{\c{c}}al, 
  On the direct decomposability of strong negations and s-implication
  operators on product lattices,
    Information Sciences 176(20) (2006) 3011--3025.

\bibitem{karacal2017aggregation}
 F. Karacal and R. Mesiar, 
  Aggregation functions on bounded lattices,
    International Journal of General Systems 46(1) (2017) 37--51.

\bibitem{klement2005logical}
 E.~P. Klement and R. Mesiar, 
    Logical, algebraic, analytic and probabilistic aspects of
  triangular norms,
  Elsevier, 2005.

\bibitem{kolman2003discrete}
B. Kolman R.~C. Busby and S. Ross, 
    Discrete mathematical structures,
  Prentice-Hall, Inc., 1995.

\bibitem{lidl2012applied}
 R.  Lidl and G. Pilz, 
    Applied abstract algebra,
  Springer Science \& Business Media, 2012.

\bibitem{ray1997representation}
  S. Ray, 
  Representation of a boolean algebra by its triangular norms,
    Mathware \& soft computing  4(1) (1997) 63--68.
    
\bibitem{roman2008lattices}
 S. Roman, 
    Lattices and ordered sets,
  Springer Science \& Business Media, 2008.
  
\bibitem{schweizer1983b}
  B. schweizer and A. sklar, probabilistic metric spaces, Courier Corporation,  2011.

\bibitem{Skala1971trellis}
  H. Skala, 
  Trellis theory,
    Algebra Universalis 1(1) (1971) 218--233.

\bibitem{Skala1972trellis}
  H. Skala, 
    Trellis theory,
  American Mathematical Soc., 1972.


\bibitem{yettou1}
  M. Yettou, A. Amroune and L. Zedam, 
  A binary operation-based representation of a lattice,
    Kybernetika 55(2) (2019) 252--272.
    

\end{thebibliography}
\end{document}